\documentclass[a4paper,10pt]{article}

\usepackage[T1]{fontenc}
\usepackage{enumerate}
\usepackage{amsmath,amssymb,amsthm}
\usepackage{graphicx}
\usepackage{bbm}
\usepackage{todonotes} 
\usepackage{cite}      
\usepackage[affil-it]{authblk} 

\usepackage[hidelinks,bookmarksnumbered]{hyperref}
\hypersetup{pdfstartview=FitH}
\usepackage{tikz}
\usetikzlibrary{patterns}
\usetikzlibrary{calc}
\usetikzlibrary{intersections}
\usetikzlibrary{through}

  \newcounter{constant}
  \newcommand{\newconstant}[1]{\refstepcounter{constant}\label{#1}}
  \newcommand{\uc}[1]{c_{\textnormal{\tiny \ref{#1}}}}
\def\arraypar#1{\parbox[c]{\textwidth - 2cm}{\centering #1}}

\usepackage{environ}
\NewEnviron{display}{
  \begin{equation}\begin{array}{c}
    \arraypar{\BODY}
  \end{array}\end{equation}
}

\newcommand{\ind}[1]{\ensuremath{\mathbbm{1}_{ \{#1 \}}}}
\newcommand{\inds}[1]{\ensuremath{\mathbbm{1}_{ #1 }}}
\newcommand{\indAlto}[1]{\ensuremath{\mathbbm{1} \{ #1 \}}}

\newcommand{\distr}{\ensuremath{\stackrel{\scriptstyle d}{=}}}
\newcommand{\supp}{\ensuremath{\text{\upshape supp }}}

\newcommand{\dist}{\ensuremath{\text{\upshape dist}}}
\newcommand{\Poi}{\text{\upshape Poi}}
\newcommand{\Var}{\text{\upshape Var}}
\newcommand{\diam}{\text{\upshape diam }}

\newcommand{\cB}{\ensuremath{\mathcal{B}}}
\newcommand{\cC}{\ensuremath{\mathcal{C}}}

\newcommand{\cE}{\ensuremath{\mathcal{E}}}

\newcommand{\cR}{\ensuremath{\mathcal{R}}}

\newcommand{\cV}{\ensuremath{\mathcal{V}}}

\newcommand{\LL}{\ensuremath{\mathbb{L}}}
\newcommand{\NN}{\ensuremath{\mathbb{N}}}

\newcommand{\QQ}{\ensuremath{\mathbb{Q}}}
\newcommand{\RR}{\ensuremath{\mathbb{R}}}
\newcommand{\ZZ}{\ensuremath{\mathbb{Z}}}

\theoremstyle{plain}
\newtheorem{prop}{Proposition}[section]
\newtheorem{conj}{Conjecture}[section]
\newtheorem{teo}{Theorem}[section]
\newtheorem{lema}{Lemma}[section]
\newtheorem{coro}{Corollary}[section]
\newtheorem*{teo:vacant_crossing_of_boxes_alpha_2}{Theorem \ref{teo:vacant_crossing_of_boxes_alpha_2}}
\newtheorem*{teo:alpha_phase_transition}{Theorem \ref{teo:alpha_phase_transition}}
\newtheorem*{teo:covered_set_alpha_phase_transition}{Theorem \ref{teo:covered_set_alpha_phase_transition}}

\theoremstyle{definition}
\newtheorem{defi}{Definition}[section]

\theoremstyle{remark}
\newtheorem{remark}{Remark}[section]

\title{Ellipses Percolation}
\author{Augusto Teixeira\thanks{Email: \texttt{augusto@impa.br}; IMPA, Estrada Dona Castorina 110, 22460-320, Rio de Janeiro, RJ, Brazil.} \quad \qquad Daniel Ungaretti\thanks{Email: \texttt{danielungaretti@gmail.com}; IMPA, Estrada Dona Castorina 110, 22460-320, Rio de Janeiro, RJ, Brazil.}}

\begin{document}

\maketitle

\begin{abstract}
  We define a continuum percolation model that provides a collection of random ellipses on the plane and study the behavior of the covered set and the vacant set, the one obtained by removing all ellipses. Our model generalizes a construction that appears implicitly in the Poisson cylinder model of Tykesson and Windisch. The ellipses model has a parameter $\alpha > 0$ associated with the tail decay of the major axis distribution; we only consider distributions $\rho$ satisfying $\rho[r, \infty) \asymp r^{-\alpha}$. We prove that this model presents a double phase transition in $\alpha$. For $\alpha \in (0,1]$ the plane is completely covered by the ellipses, almost surely. For $\alpha \in (1,2)$ the vacant set is not empty but does not percolate for any positive density of ellipses, while the covered set always percolates. For $\alpha \in (2, \infty)$ the vacant set percolates for small densities of ellipses and the covered set percolates for large densities. Moreover, we prove for the critical parameter $\alpha = 2$ that there is a non-degenerate interval of density for which the probability of crossing boxes of a fixed proportion is bounded away from zero and one, a rather unusual phenomenon. In this interval neither the covered set nor the vacant set percolate, a behavior that is similar to critical independent percolation on $\ZZ^2$.

\vspace{0.5cm}\noindent {\it Math. Subject Classification:} 60K35, 82B43, 60G55.
\end{abstract}

\section{Introduction}

Bernoulli percolation was introduced by Broadbent and Hammersley \cite{broadbent1957percolation} in 1957 and is a simple model that exhibits phase transition. Since then, many interesting properties of the different phases have been well understood and there are classical books on the subject \cite{grimmett2010percolation, bollobas2006percolation}. However, some important problems remain open and the model continues to attract the attention of the probability community.

Percolation processes are natural candidates to model environments in which connectivity is assumed to be random and thus found many applications in different areas such as forest fires \cite{drossel1992self, brouwer2005percolation}, spread of infections \cite{newman2002spread} and polymerization \cite{aldous2000percolation}, among others.

Continuum percolation models \cite{meester1996continuum} are one of the possible variations. Some techniques from percolation on graphs work on this setting, although working in a continuous space generally adds some extra difficulties. Moreover, there are results on continuum percolation that have no graph analogues, which increases interest on this subject. For example, in \cite{hall1985continuum} it is proven that there are Boolean models with random radii such that the critical point for the existence of an infinite cluster and the critical point for the mean cluster size to be infinite are different.

One of the most established continuum percolation models is the Boolean model, in which we start with a Poisson point process in $\RR^d$ with intensity measure being a multiple of the Lebesgue measure and for each point we add a ball (possibly of random radius) centered on it, independently for each point. In this article we define a similar model that provides a collection of random ellipses in the plane. Our model is inspired by a paper of Tykesson and Windisch \cite{1010.5338}, in which they have defined the Poisson cylinder model on $\RR^d$. Many of their results follow by looking at the intersection of the collection of cylinders with a plane $\RR^{2} \times \{0\}^{d-2}$; by performing this intersection one obtains ellipses.

Let us define our model, a generalization of the random ellipses obtained on \cite{1010.5338}. We build a collection of ellipses that are centered on a Poisson point process on $\RR^2$ of intensity $u$ times Lebesgue measure, where $u > 0$. The ellipses have uniform direction and the size of their minor axis is always equal to one. Moreover, their major axis has distribution $\rho$ supported on $[1, \infty)$ and satisfying
\begin{equation*}
\uc{const:R_decay}^{-1} r^{-\alpha} \le \rho[r,\infty) \le \uc{const:R_decay} r^{-\alpha}, \ \text{for every} \ r \geq 1
\end{equation*}
for some positive constant $\uc{const:R_decay}$. We refer to this process as the $(u, \rho)$-ellipses model. Details of the construction are given in the next section.

\begin{figure}[h]
    \centering
    \includegraphics[width=0.6\textwidth]{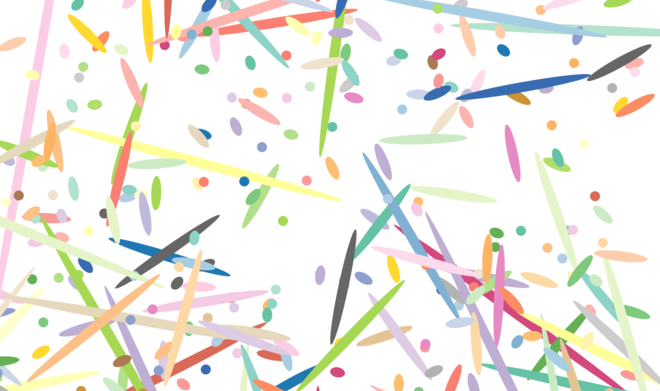}
    \caption{Small simulation of ellipses model with $\rho[r,\infty) = r^{-2}$.}
    \label{fig:leandro_simulation}
\end{figure}

One reason to study this model is that it presents infinite range dependencies, which prevents the use of some common tools like Peierls argument. Also, although the ellipses model dominates Boolean model with radius 1, it cannot be dominated by any Boolean model with fixed radius (see Remark \ref{remark:not_dominated_by_Boolean} after Proposition \ref{prop:prob_small_ball_covered_by_ellipse_step2}). On the other hand, any $(u, \rho)$-ellipses model is obviously dominated by a Boolean model with radius distribution $\rho$ and this can be used to derive some of the results in Theorems \ref{teo:alpha_phase_transition} and \ref{teo:covered_set_alpha_phase_transition}. To overcome the dependencies of the model, our study of ellipses percolation uses similar techniques as other models with long range dependencies such as the Poisson cylinder model of \cite{1010.5338} and the random interlacements model of Sznitman \cite{sznitman2010vacant}. 

We are mainly interested in the phase transition properties that the model presents. We prove that the vacant set $\cV$, the set not covered by any ellipse, presents a double phase transition in $\alpha$.

\begin{teo}
\label{teo:alpha_phase_transition}
Consider $(u, \rho)$-ellipses percolation model, where $\rho$ has tail decay $\alpha$ and associated constant $\uc{const:R_decay}$. Then, with probability one:
\begin{enumerate}[1.]
 \item If $\alpha \in (0,1]$ we have $\cV = \varnothing$ for every $u > 0$.\vspace{-2mm}
 \item If $\alpha \in (1,2)$ we have $\cV \neq \varnothing$, but for any $u > 0$ there is no percolation.\vspace{-2mm}
 \item If $\alpha \in (2,\infty)$ there exists a critical value $\bar{u}_c(\rho) \in (0, \infty)$ such that if $u < \bar{u}_c$ then $\cV$ percolates and if $u > \bar{u}_c$ then $\cV$ does not percolate.
\end{enumerate}
\end{teo}

We also prove a version of Theorem \ref{teo:alpha_phase_transition} for the covered set $\cE$, providing an overall picture of an ellipses model.

\begin{teo}
\label{teo:covered_set_alpha_phase_transition}
Consider $(u, \rho)$-ellipses percolation model, where $\rho$ has tail decay $\alpha$ and associated constant $\uc{const:R_decay}$. Then, with probability one:
\begin{enumerate}[1.]
 \item If $\alpha \in (0,1]$ we have $\cE = \RR^2$ for every $u > 0$.\vspace{-2mm}
 \item If $\alpha \in (1,2)$ we have that $\cE$ percolates for any $u > 0$.\vspace{-2mm}
 \item If $\alpha \in (2,\infty)$ there exists a critical value $u_c(\rho) \in (0, \infty)$ such that if $u < u_c$ then $\cE$ does not percolate and if $u > u_c$ then $\cE$ percolates.
\end{enumerate}
\end{teo}

The case $\alpha = 2$ is described separately in Theorem \ref{teo:vacant_crossing_of_boxes_alpha_2} because this case deserves special attention. It presents an unusual phase transition in $u$.

\begin{teo}
\label{teo:vacant_crossing_of_boxes_alpha_2}
Let $\rho$ be a distribution with $\alpha = 2$. Then, there exists $\bar{u} = \bar{u}(\uc{const:R_decay}) > 0$ such that for any fixed $k > 0$, $u \in (0,\bar{u})$ and $l > 0$ 
\begin{equation}
\label{eq:vacant_crossing_bounded_away_intro}
\delta \le P_{u, \rho}[\text{exists vacant horizontal crossing of box of height $l$ and width $kl$}]\le 1 - \delta,
\end{equation}
where $\delta = \delta(\uc{const:R_decay}, u, k) > 0$. Moreover, for $u \in (0, \bar{u})$ we have:
\begin{equation}
\label{eq:cV_and_cE_not_percolate_intro}
P_{u, \rho}[\text{neither} \ \cV \ \text{nor} \ \cE \ \text{percolate}] = 1.
\end{equation}
\end{teo}

In other words, equation \eqref{eq:vacant_crossing_bounded_away_intro} states that the probability of having vacant crossing of boxes is bounded away from zero and one, independently of the scale of the box. This property holds for an interval $(0, \bar{u})$, not only a point. Together with equation \eqref{eq:cV_and_cE_not_percolate_intro}, Theorem \ref{teo:vacant_crossing_of_boxes_alpha_2} shows some similarity between critical bond percolation on $\ZZ^2$ and ellipses models with $\alpha = 2$ and $u$ on $(0, \bar{u})$.

The existence of a non-trivial interval $(0, \bar{u})$ in which the model features non-degenerate crossing probabilities is very interesting. This result has the same flavor of some other phenomena already in the literature. In \cite{schrenk2016critical} it is proven for a random fragmentation model with long-range correlations that there is an entire off-critical region in which power-law scaling is observed. Another example can be found on Coordinate Percolation on $\ZZ^3$ \cite{hilario2011coordinate}; in this model, each column that is parallel to one of the coordinate axis of $\ZZ^3$ is removed or not with a probability parameter depending only on its direction and columns are removed or not independently. This model has infinite range dependencies. In \cite{hilario2011coordinate} it is shown that the tail distribution for the radius of the open cluster containing the origin decays exponentially fast when at least two of the parameters are fixed to be high, but if two of the parameters are taken relatively small, then the truncated version for this tail decays, at most, polynomially fast. Quoting reference \cite{schrenk2016critical}, ``these findings suggest that long-range directional correlations lead to a rich spectrum of critical phenomena which need to be understood''.

Theorems \ref{teo:alpha_phase_transition} and \ref{teo:covered_set_alpha_phase_transition} have statements that are quite similar. Looking closely at them, one can identify that when $\alpha \in (0,2)$ the model is somewhat trivial and when $\alpha > 2$ there is a phase transition in $u$. Notice that the critical points $\bar{u}_c(\rho)$ and $u_c(\rho)$ mentioned on Theorems \ref{teo:alpha_phase_transition} and \ref{teo:covered_set_alpha_phase_transition} do not need to be equal. However, we believe it holds

\begin{conj}
\label{conj:equality_critical_points_alpha_greater_2}
When $\alpha > 2$ we have $\bar{u}(\rho) = u_c(\rho)$.
\end{conj}

This would imply that when $\alpha > 2$ the phase transition in $u$ is rather classical, despite the long-range dependencies. In the beginning of Section \ref{section:phase_transition_for_existence_of_critical_point} we define both $\bar{u}_c$ and $u_c$ and discuss their relation more deeply (see remark \ref{remark:properties_of_uc_and_bar_uc}).

Let us discuss the ideas of the proofs and the main tools used in the paper. Our model can be defined as a Poisson point process on a larger space (see Definition \ref{defi:ellipses_model}). Thus, we are able to estimate the probability of many useful events by making an appropriate thinning of it, as described in Proposition \ref{prop:thinning} below. These estimates are in the core of many proofs. 

We provide two proofs that the plane is completely covered by ellipses iff $\alpha \le 1$. The first is a consequence of the estimates of Section \ref{section:probability_of_simple_events} and Borel-Cantelli lemma. The second makes use of an argument in Hall \cite{hall1985continuum} to relate total covering to the expected area of an ellipse being infinite.

The proof that $\cE$ percolates for every $u >0$ when $\alpha \in (1,2)$ follows from bounds on the probability of having a left-right covered crossing of a box done by exactly one ellipse. We build a sequence of nested boxes with the property that if we have covered crossings for all but finitely many of them we guarantee $\cE$ percolates. 

In order to prove that for $\alpha \in (1,2)$ the set $\cV$ does not percolate for any $u>0$, we adapt the proof of Proposition 5.6 from \cite{1010.5338}; we prove that with probability one there is an infinite number of circuits made of exactly three ellipses surrounding the origin, by a Borel-Cantelli argument.  

We finish Theorems \ref{teo:alpha_phase_transition} and \ref{teo:covered_set_alpha_phase_transition} by proving that when $\alpha > 2$ is fixed there is a phase transition in $u$ for the percolation of $\cV$ and also of $\cE$. This is done by dominating the ellipses model with a Boolean model with radius distribution $\rho$ (see \cite{gouere2008subcritical}). The phase transition for $\cE$ follows directly from this domination, but for $\cV$ we need to develop some additional arguments since \cite{gouere2008subcritical} does not study the vacant set. 

The proof of property \eqref{eq:vacant_crossing_bounded_away_intro} in Theorem \ref{teo:vacant_crossing_of_boxes_alpha_2} uses the estimates derived in the previous sections together with a coupling with fractal percolation, also known as Mandelbrot percolation \cite{ChayesDurrett1988fractal_percolation}. Our coupling uses the results of Ligget, Schonmann and Stacey \cite{liggett1997domination}. Finally, to conclude that $\cE$ does not percolate for small $u$ we use \eqref{eq:vacant_crossing_bounded_away_intro} together with a bound on decay of correlations and a generalization of Borel-Cantelli lemma from \cite{ortega1984sequence}, needed to deal with events that are not too far from being independent.

There are still some interesting unanswered questions regarding ellipses model. One of them is already stated as Conjecture \ref{conj:equality_critical_points_alpha_greater_2}. Another question is to understand better what actually happens when $\alpha = 2$. We only showed the existence of a phase transition in $u$ related to the probability of crossing boxes of fixed ratio, but it is possible that there are other phase transitions. For instance, one could define
\begin{equation*}
u_{\text{cross}}(\rho) := \sup\{\bar{u}; \ \text{\eqref{eq:vacant_crossing_bounded_away_intro} holds}\} \ \text{and} \ u_{\text{exp}}(\rho) := \inf\{u; \ P_{u,\rho}[0 \overset{\cV}{\leftrightarrow} \partial B(n)] \ \text{decays exponentially}\}
\end{equation*}
and check if any of them coincide with $u_c(\rho)$. Finally, it would be interesting to say anything about what happens in the critical point when $\alpha = 2$.

This paper is divided in seven Sections. In Section \ref{section:description_of_the_model} we define the ellipses model rigorously and prove that restricting the Poisson cylinder model of \cite{1010.5338} can be seen as a particular case of the ellipses model. Section \ref{section:probability_of_simple_events} collects estimates on the probability of useful events that are used in the next sections. Section \ref{section:phase_transition_for_total_covering} proves that the ellipses cover the whole plane iff $\alpha \le 1$; it also contains a slight generalization of total covering for other continuum percolation models on the plane and a law of large numbers for the number of ellipses covering a small euclidean ball $B(\varepsilon)$ when $\varepsilon \in [0,1/2)$ is fixed, according to the position of their centers. Section \ref{section:phase_transition_for_existence_of_critical_point} finishes the proof of Theorems \ref{teo:alpha_phase_transition} and \ref{teo:covered_set_alpha_phase_transition}, Section \ref{section:decay_of_correlations_with_distance} provides tools to bound the correlation of some events and proves that translations are ergodic for the ellipses model and Section \ref{section:vacant_crossing_of_boxes_for_alpha_equal_2} proves Theorem \ref{teo:vacant_crossing_of_boxes_alpha_2}. 

\vspace{2mm}
\textbf{Acknowledgments:} The authors would like to thank Caio Alves and Serguei Popov for insights on how to find the major axis distribution for the ellipses model derived from Poisson cylinder model. Also, we thank Leandro Cruz for Figure \ref{fig:leandro_simulation}. This work had financial support from CNPq grants 306348/2012-8 and 478577/2012-5, FAPERJ by grants 202.231/2015 and 200.195/2015 and also from Capes.

\section{Description of the Model}
\label{section:description_of_the_model}

\vspace{4mm}
We begin defining a model that provides a random collection of ellipses on $\RR^2$. This process will be referred to as the \textit{ellipses model}. To build it, we need three ingredients:
\begin{enumerate}
 \item Denote by $\lambda$ the Lebesgue measure on $\RR^2$. Given $u \in (0, \infty)$, we consider a Poisson Point Process (PPP) on $\RR^2$ with intensity $u \lambda$ which we denote by $\omega = \sum_i \delta_{x_i}$, where $\{x_i\} \subset \RR^2$ is countable and locally finite. A concise notation for this definition is $\omega \distr \text{PPP}(u \lambda)$. A reference for PPP's can be found on \cite{resnick2013extreme}.

 \item Given $\alpha > 0$, let $\rho$ be a distribution on $[1,\infty)$ such that $P[R \geq r] \asymp r^{- \alpha}$ for $r \geq 1$, i.e., there is a constant $\uc{const:R_decay} = \uc{const:R_decay}(\rho) > 0$ such that 
\newconstant{const:R_decay} 
\begin{equation}
\uc{const:R_decay}^{-1} r^{-\alpha} \le \rho[r,\infty) \le \uc{const:R_decay} r^{-\alpha}, \ \forall r \geq 1
\label{eq:def_of_R}
\end{equation}
 
 \item A random variable $V$ with distribution $V \distr \text{U}(0,\pi)$. The law of $V$ will be denoted by $\nu$.
\end{enumerate}

Define $S = \RR^2 \times [1, \infty) \times (- \pi/2, \pi/2] \subset \RR^4$.
 
\begin{defi}
\label{defi:ellipses_model}
The \textit{ellipses model} is a PPP on $S$ with intensity measure given by the product 
$(u\lambda) \otimes \rho \otimes \nu$. We denote it by $\xi = \sum_i \delta_{s_i}$.
\end{defi}

Let us see how the ellipses model can be seen as an actual collection of ellipses on $\RR^2$. 
Whenever we say ellipse, we mean the curve described by the ellipse together with its interior. 
For an element $(z, R, V) \in S$, we define $E(z, R, V)$ as the ellipse with center $z$ 
and major axis of size $R$, minor axis of size 1 and direction $v$. In this way, $E$ is a function from $S$ to 
the subsets of $\RR^2$ that provides a natural identification of $S$ and the ellipses we 
work with. We denote by $\cE := \cE(\xi)$ the random subset of $\RR^2$ formed by the 
union of ellipses given by the PPP on $\RR^4$:

$$
\cE(\xi) = \bigcup_{s \in \supp \xi} E(s).
$$

We also write $\cV := \RR^2 \backslash \cE$. The sets $\cE$ and $\cV$ will be called 
the \textit{covered} and \textit{vacant} sets, respectively. Our greatest concern with 
the ellipses model will be to understand the connectivity behavior of the area covered by the 
ellipses and its complement, as we change the parameters of the model.

We highlight two important parameters of the ellipses model:

\begin{itemize}
 \item The parameter $\alpha$ which controls the tail of the distribution of the ellipses' major axes. As its value grows, the size of the major axes tends to have lighter tails.
 
 \item The parameter $u$, which controls the intensity of the Poisson Point Process.
\end{itemize}

We will refer to the ellipses model by writing $(u, \rho)$-ellipses model. When working with the associated probability space $P_{u, \rho}$, we sometimes omit the dependence on these parameters if there is no danger of confusion.

\begin{remark}
\label{remark:invariances_of_ellipses_model}
It is worth mentioning that our model has translational and rotational invariance. Indeed, this follows from two facts. The first is that the Lebesgue measure on $\RR^2$ has rotational and translational invariance. The second one is our choice of uniform  distribution for the directions of the ellipses.
\end{remark}

\vspace{2mm}
\noindent\textbf{Notation:} We constantly use the following notation. For denoting boxes in $\RR^2$, let %
$$B_{\infty}(l;k) = [-lk/2, lk/2] \times [-l/2, l/2].$$%
Denote by $L^{-}(l;k)$ and $L^{+}(l;k)$ the left and right sides of box $B_{\infty}(l;k)$; that is, the sets $\{ -lk/2 \} \times [-l/2, l/2]$ and $\{ +lk/2 \} \times [-l/2, l/2]$, respectively. We denote the euclidean ball on $\RR^d$ with center on a point $w$ and radius $r$ by $B(w,r)$.

Given an ellipse $E = E(s)$ with $s \in S$, it is useful to be able to recover its defining parameters; we define $c(E)$ as the point in $\RR^2$ that is the center of the ellipse. Define also $R(E)$ as the size of the major axis of $E$ and $V(E)$ as the direction of its major axis, in the interval $(-\pi/2, \pi/2]$. In the specific case where $R(E) = 1$, we will not be able to recover $V(E)$, but this has no relevance in this work. Using the identification we already mentioned, these functions are obtained by composing $E^{-1}$ and a projection defined on $S$.

We also add a short note on our notation for constants. Constants that appear during calculations are generally denoted by $c$ or $C$ and can change from line to line. However, for more important constants we add a subscript number referring to their first appearance in the text.

\subsection{Relation with Poisson Cylinder Model}
\label{subsection:relation_with_poisson_cylinder_model}

The idea of the ellipses model comes from an article of Tykesson and Windisch \cite{1010.5338}. In \cite{1010.5338}, they define a Poisson cylinder model in $\RR^d$ and study whether the vacant set, the one obtained after removing all cylinders, percolates or not. The ellipses come up when we restrict our attention to the intersection of the cylinders and $\RR^2 \times \{ 0 \}^{d-2}$.

Clearly, percolation of the vacant set in a hyperplane implies percolation in all space $\RR^d$. This simplification is important in their proof that for $d \geq 4$ there is a non-trivial phase transition in parameter $u$. However, when $d=3$, a different phenomenon takes place. Tykesson and Windisch show that, $\forall u > 0$, there is  an infinite number of circuits of ellipses surrounding the origin with probability one. Hence, there can't be percolation of the vacant set restricted to a plane.

Returning to the ellipses model, we can link it with the Poisson cylinder model for dimension $d=3$ through 
the following proposition:

\begin{prop}
\label{prop:relation_Poisson_cylinder_model}
The Poisson cylinder model for $d=3$ restricted to a plane is equivalent to an ellipses model, when we take $\rho(r, \infty) = r^{-2}$ for $r \geq 1$.
\end{prop}

Proposition \ref{prop:relation_Poisson_cylinder_model} states which is the distribution $\rho$ for this specific model. However, the original definition of how to construct the random cylinders in \cite{1010.5338} does not provide us with this distribution explicitly. Therefore, we now present one way to obtain it.

Let $e_3 = (0,0,1) \in \RR^3$ and denote the usual inner product on $\RR^3$ by $\langle \cdot \, , \cdot \rangle$. Also, denote by $D$ the upper half of $S^2$, that is $D = \{ w \in \RR^3 : \ ||w|| = 1 \text{ and } \langle w \, , e_3 \rangle > 0 \}$. Consider the following construction:
\begin{enumerate}[C.1)]
 \item Start with a PPP $\gamma$ on the plane $\RR^2 \times \{0\}$ with intensity measure $u \lambda$, where $\lambda$ is the Lebesgue measure. The points on $\supp \gamma$ will be the points where the axes of the cylinders intersect the plane.
 \item Independently for every point $r \in \supp \gamma$, choose a random direction on $D$. The distribution of this random direction will be denoted by $\beta$, and each $r \in \supp \gamma$ has its own $\beta_r$, independent of everything else. The lines passing through $r$ and with direction $\beta_r$ for $r \in \supp \gamma$ will be axes of the cylinders. 
\end{enumerate}

We would like to know what should be the distribution of $\beta$ in order for this alternative construction to be equivalent to the construction in \cite{1010.5338}; $\rho$ is easily obtained from $\beta$. A first guess on $\beta$ could be the uniform distribution on $D$. However, this guess is wrong.

Let $\mu$ be a non-trivial measure on $\LL$ which is invariant with respect to rotations and translations. We begin by identifying the space $\LL$ of lines in $\RR^3$ with the space $D \times \RR^2$. More precisely, we do not identify all space $\LL$, but instead we work with $\LL^\ast$, the space of all lines in $\RR^3$ which are not contained in some plane $\RR^2 \times \{z\}$. As we will see, the fact that our measure $\mu$ has rotational and translational invariance implies that $\mu(\LL\backslash\LL^\ast) = 0$ and thus $\LL\backslash\LL^\ast$ can be ignored (see Remark \ref{remark:ignore_LL_backslash_LLast} below Proposition \ref{prop:intensity_measure_TW_model}).

Notice that any line $l \in \LL^\ast$ intersects the plane $\RR^2 \times \{0\}$ at exactly one point; denote its $(x,y)$ coordinates by $p(l)$. Also, every $l$ has an unequivocal direction in $D$, denoted by $d(l)$. Consider the function $\Phi : \LL^\ast \to D \times \RR^2 \text{ such that } l \mapsto (d(l), p(l))$. It is clear that $\Phi$ is a bijection. So, in order to know the measure $\mu$ restricted to $\LL^\ast$ we just need to understand what is the induced measure on the space $D \times \RR^2$, and we denote this measure by $\tilde{\mu}$.

For any $w \in D$ define $\psi(w)$ as the angle between $e_3$ and $w$; this can be written as $\psi(w) = \arccos \langle e_3, w \rangle$. We state Proposition \ref{prop:intensity_measure_TW_model} without providing a proof. This result was already known, although not exactly in this formulation (see \cite{kendall1963geometrical}, pages 93-100). In \cite{ungaretti2017ellipses_percolation} we provide a complete proof, based on discussions with Caio Alves and Serguei Popov.

\begin{prop}
\label{prop:intensity_measure_TW_model}
Let $\mu$ be a non-trivial measure on $\LL$ which is translational and rotational invariant and let $\tilde{\mu}$ be its induced measure on the space $D \times \RR^2$ through the function $\Phi$. Then $\tilde{\mu} = \phi \otimes \lambda$ where $\lambda$ is Lebesgue measure on $\RR^2$ and 

$$
\phi(A) = c \int_A \cos\psi(w) \,\sigma(\text{d}w) \text{ for } A \in \cB(D),
$$

\noindent with $\sigma$ being the uniform measure on $D$ and $c >0$ a universal constant.
\end{prop}

\begin{remark}
\label{remark:ignore_LL_backslash_LLast}
Proposition \ref{prop:intensity_measure_TW_model} explicits the intensity measure for the PPP restricted to $\LL^\ast$. Using this representation and rotational invariance we can conclude that $\mu(\LL\backslash\LL^\ast) = 0$. Indeed, if $\cR$ is any rotation that does not leave the $xy$ plane invariant then $\{d(l); l \in \cR(\LL\backslash\LL^\ast)\}$ is  the intersection of $D$ with a plane and must have $\sigma$ measure zero.
\end{remark}

\begin{proof}[Proof of Proposition \ref{prop:relation_Poisson_cylinder_model}]
The proof is immediate from Proposition \ref{prop:intensity_measure_TW_model}, once we notice that $\rho(r, \infty) = P[R \geq r] = \beta(\{w \in D; \tfrac{1}{\cos \psi(w)} \geq r\})$.
\end{proof}

\section{Probability of Simple Events}
\label{section:probability_of_simple_events}

We would like to prove some properties of the ellipses model which are analogous to classical results in percolation. For this purpose, it is useful to have in hands estimates on the probability of some simpler, more basic events. These events will be used later to build more complex ones. This section is devoted to collecting these estimates.

We want to estimate the probability of the intersection and covering events
\begin{equation}
\{ \cE \cap A \neq \varnothing \} \ \text{and} \ \{ A \subset \cE \}
\label{eq:simple_events}
\end{equation}
where $A \subset \RR^2$ is some fixed set. Such estimates are done in two steps. First, we fix a point $z \in \RR^2$ and try to bound the probability of a random ellipse centered at $z$ to intersect (or cover) $A$. Of course, this probability will depend on $z$ and $A$, and we would like that this dependence is not too complicated for calculations. For that reason, we will only be concerned with sets $A$ which are reasonably simple, such as points, segments and balls.

The second step consists in taking into consideration the positions of all the centers of ellipses on the support of $\omega$, the PPP on $\RR^2$. This can be studied as a \textit{thinning} of the PPP $\omega$. We use a proposition that can be found, for instance, in Meester and Roy \cite{meester1996continuum}, Proposition 1.3.  We adapted their version to our notational conventions:

\begin{prop}
\label{prop:thinning}
Let $\omega$ be a PPP on $\RR^d$ with intensity measure $u\lambda$ and $g$ be a measurable function $g:~\RR^d~\to~[0,1]$. Define $\omega_g$ the thinning of $\omega$ that keeps every point $z \in \supp \omega$ independently with probability $g(z)$. Then $\omega_g$ is a non-homogeneous PPP on $\RR^d$ with intensity measure $\mu_g$ given by 
\begin{equation*}
\mu_g(U) = u\int_{U}g(z)\, \text{d}z.
\end{equation*}
\end{prop}

Before stating our results, we notice that the invariances of the process (see Remark \ref{remark:invariances_of_ellipses_model}) can be used to make our task easier. We can apply any rigid motion to $A$  without altering the probability of the events on \eqref{eq:simple_events}.

\subsection{Estimates for covering a small ball}

We investigate the probability of covering an euclidean ball $B(w, \varepsilon)$ for $0 \le \varepsilon < 1/2$. Notice that we allow $\varepsilon = 0$ i.e., $B(w, 0) = \{w\}$.

\newconstant{const:ellipse_covers_small_ball}
\begin{lema}
\label{lema:ellipse_covers_small_ball_step1}
Let $w, z \in \RR^2$ and $0 \le \varepsilon < \tfrac{1}{2}$. Let $E_z$ be a random ellipse centered on $z$ and whose major axis is distributed as $\rho \otimes \nu$. Then, there are constants $\uc{const:ellipse_covers_small_ball} = \uc{const:ellipse_covers_small_ball}(\uc{const:R_decay}, \alpha) > 0$ and $c(\varepsilon) \geq 2$ such that for $|z - w| > c(\varepsilon)$ we have
\begin{equation*}
\uc{const:ellipse_covers_small_ball}^{-1} |z - w|^{-(\alpha + 1)} 
\le P(B(w, \varepsilon) \subset E_z) 
\le \uc{const:ellipse_covers_small_ball} |z - w|^{-(\alpha + 1)}.
\end{equation*}
\end{lema}

\begin{proof}
Without loss of generality, we may assume $w$ is the origin and $z = (0, |z|)$. Thus, we have to prove
$\uc{const:ellipse_covers_small_ball}^{-1} |z|^{-(\alpha + 1)} \le P(B(\varepsilon) \subset E_z) \le \uc{const:ellipse_covers_small_ball} |z|^{-(\alpha + 1)}$.

By symmetry, we can state that $P[0 \in E_z] = P[z \in E_0]$, where $E_0$ is a random ellipse centered in the origin with major axis size and direction given by $R$ and $V$. Rotate the ellipse $E_0$ by the angle $V$ clockwise and denote this rotation by $\cR_V$. This rotation sends the ellipse $E_0$ to the ellipse $\cR_V(E_0)$. In particular, the major axis of $E_0$ is sent to the horizontal position and the minor axis, to the vertical one. Hence, the equation of $\cR_V(E_0)$ in the plane is given by $\{ (x,y) ; \left( x/R \right)^2 + y^2 = 1 \}$. Furthermore, the region covered by this ellipse is made of points with $\left( x/R \right)^2 + y^2 \le 1$. Then:
\begin{align}
P(z \in E_0)
 &= P( \cR_V(z) \in \cR_V(E_0)) = P\left( \left(\frac{|z| \cos V}{R}\right)^2 + (|z| \sin V)^2 \le 1 \right) \nonumber\\
 &= P\left( \left(\frac{ \cos V}{R}\right)^2 + \sin^2 V \le \frac{1}{\ |z|^2} \right) = P\left( \left(\frac{1}{R}\right)^2 + \sin^2 V \left[ 1 - \frac{1}{R^2} \right]\le \frac{1}{\ |z|^2} \right).
\label{eq:prob_z_in_E0_as_integral_in_R_and_v}
 \end{align}

\vspace{3mm}
\noindent\textbf{Upper Bound:} For the upper bound we notice that $P(B(\varepsilon) \subset E_z) \le P(0 \in E_z)$. We can assume $|z| > 2$, since we will choose $c(\varepsilon) \geq 2$. In order to estimate the probability in \eqref{eq:prob_z_in_E0_as_integral_in_R_and_v} we notice that
\begin{equation}
\left(\frac{1}{R}\right)^2 + \sin^2 V \left[ 1 - \frac{1}{R^2} \right]\le \frac{1}{\ |z|^2} \ \Longrightarrow \ \frac{1}{R^2} \le \frac{1}{\ |z|^2} \ \Longrightarrow \ R \geq |z|
\label{eq:lemma_A_point_upper_bound_1}
\end{equation}
and also, by the same reasoning and \eqref{eq:lemma_A_point_upper_bound_1} we have
\begin{equation}
\sin^2 V \le \frac{1}{\ |z|^2} \left[ 1 - \frac{1}{R^2}\right]^{-1} \le \frac{1}{\ |z|^2} \left[ 1 - \frac{1}{\ |z|^2}\right]^{-1} = \frac{1}{ |z|^2 - 1 }.
\label{eq:lemma_A_point_upper_bound_2}
\end{equation}

Equations \eqref{eq:lemma_A_point_upper_bound_1} and \eqref{eq:lemma_A_point_upper_bound_2} imply that the event whose probability we want to estimate is contained into the rectangular event $\{ R \geq |z|, \ |\sin V| \le (|z|^2 - 1)^{-1/2} \}$. Then, we can use the independence between $R$ and $V$ to deduce
\begin{align*}
P(0 \in E_z) 
  & \le \uc{const:R_decay} |z|^{-\alpha} \left[ \frac{2}{\pi} \arcsin\left( \frac{1}{(|z|^2 - 1)^{\frac{1}{2}}} \right) \right] \le \frac{4 \uc{const:R_decay} }{\pi} |z|^{-(\alpha + 1)} \frac{|z|}{\sqrt{|z|^2 - 1}} \le \uc{const:ellipse_covers_small_ball} |z|^{-(\alpha + 1)}
\end{align*}
where we have used that $\arcsin x \le 2 x$ for $x \in [0, 1]$ and $|z| [\, |z|^2 - 1\, ]^{-1/2} \le 2$ for $|z| > 2$ and defined the constant $\uc{const:ellipse_covers_small_ball}(\uc{const:R_decay}) = \tfrac{8 \uc{const:R_decay}}{\pi}$.

\vspace{3mm}
\noindent\textbf{Lower Bound:} For the lower bound, we notice that $P(B(\varepsilon) \subset E_z) = P(B(z,\varepsilon) \subset E_0)$, where $E_0$ is an ellipse centered on the origin with major axis' size and direction given by $R$ and $V$ respectively. Once again, we apply a clockwise rotation $\cR_V$ to get
\begin{equation*}
P(B(z,\varepsilon) \subset E_0) = P( \cR_V(z) + B(\varepsilon) \subset \cR_V(E_0)).
\end{equation*}

Now, notice that $\cR_V(z) = (|z|\cos V, - |z| \sin V)$ and for any $r = (r_1, r_2) \in B(\varepsilon)$ we have
\begin{equation*}
\{ \cR_V(z) + r \in \cR_V(E_0) \} = \left\{ \left(\tfrac{|z| \cos V + r_1}{R}\right)^2 + (|z|\sin V + r_2)^2 \le 1  \right\}.
\end{equation*} 

Using the inequality $(a + b)^2 \le 2(a^2 + b^2)$ twice, we can write
\begin{align*}
\left[\frac{|z| \cos V + r_1}{R}\right]^2 \!\! + (|z|\sin V + r_2)^2 
  &\le 2 \left(\frac{|z|^2 \cos^2 V + r^2_1}{R^2}\right) + 2(|z|^2\sin^2 V + r^2_2) \nonumber\\
  &\le 2 |z|^2 \left[ \frac{1}{R^2} + \left( 1 - \frac{1}{R^2} \right)\sin^2 V \right] + 2 \varepsilon^2 \left[ 1 + \frac{1}{R^2} \right].
\end{align*}

Finally, notice that if we make both terms of the last sum smaller than $1/2$ we guarantee the event 
$\{  \cR_V(z) + B(\varepsilon) \subset \cR_V(E_0) \}$ happens. It is easily checked that 
\begin{align*}
R \geq 2|z| , \ |V| \le \arcsin \left( \frac{1}{2|z|} \right) \ &\text{imply} \ 2 |z|^2 \left[ \frac{1}{R^2} + \left( 1 - \frac{1}{R^2} \right)\sin^2 V \right] \le \frac{1}{2}. \\
\varepsilon < 1/2, \  R \geq \frac{2\varepsilon}{\sqrt{1 - 4\varepsilon^2}} \ &\text{imply} \ 2 \varepsilon^2 \left[ 1 + \frac{1}{R^2} \right] \le \frac{1}{2}.
\end{align*}

Thus, it suffices to take $\varepsilon < 1/2$, $c(\varepsilon) = \frac{\varepsilon}{\sqrt{1 - 4\varepsilon^2}} \vee 2$ and notice
\begin{align*}
P\left[ R \geq 2|z|,\ |V| \le \arcsin \left( \tfrac{1}{2|z|} \right)\right] 
  &\geq \uc{const:R_decay}^{-1} (2|z|)^{-\alpha} \cdot 2\pi^{-1} \arcsin \left( \tfrac{1}{2|z|} \right) \geq \uc{const:ellipse_covers_small_ball}^{-1} |z|^{-(\alpha + 1)}
\end{align*}
for some constant $\uc{const:ellipse_covers_small_ball} = \uc{const:ellipse_covers_small_ball}(\uc{const:R_decay}, \alpha) > 0$ and $|z| \geq c(\varepsilon)$.
\end{proof}

Having Lemma \ref{lema:ellipse_covers_small_ball_step1}, we proceed in our two step strategy.

\begin{prop}
\label{prop:prob_small_ball_covered_by_ellipse_step2}
Let $w \in \RR^2$ and $0 \le \varepsilon < 1/2$. Then, $P[B(w,\varepsilon) \subset \cE] = 1$ if and only if $\alpha \le 1$.
\end{prop}

Notice that the result in Proposition \ref{prop:prob_small_ball_covered_by_ellipse_step2} does not depend on $u$.

\begin{proof}
We begin assuming that $\varepsilon = 0$. By translation invariance, we may assume that $w$ is the origin. Define the function $g(z)= P[0 \in E_z]$. Then, Lemma~\ref{lema:ellipse_covers_small_ball_step1} provides the asymptotic behavior of $g(z)$ as $z \to \infty$: $g(z) \asymp |z|^{-(\alpha + 1)}$. We use Proposition \ref{prop:thinning} with this function. Notice that 
\begin{equation}
\begin{split}
P[0 \notin \cE] 
   &= P[0 \in \cV] = P[\omega_g(\RR^2) = 0] = \exp\left[ - u \int_{\RR^2} g(z) \, \text{d}z\right] = 0 \\
   &\text{if and only if}\quad\int_{\RR^2} g(z) \, \text{d}z = \infty.
\end{split}
\label{eq:equiv_conv_prop_prob_vacant_origin}
\end{equation}

Since $g(z) \in [0,1]$ for all $z$, the integral in \eqref{eq:equiv_conv_prop_prob_vacant_origin} is infinite iff $g(z)$ decays to zero sufficiently slow. Using the asymptotic expression for $g$ and integrating using polar coordinates
\begin{equation}
\int_{\RR^2} g(z) \, \text{d}z = \infty \
	\text{if and only if} \ \int_{c}^{\infty} r^{-\alpha} \, \text{d}r = \infty, \ 
	\text{if and only if} \ \alpha \le 1.
\end{equation}

Now, we handle the case $0 < \varepsilon < 1/2$. If $\alpha \le 1$, then the same argument above with the function $g(z) := P[B(\varepsilon) \subset E_z]$ shows that $P[B(\varepsilon) \subset \cE] = 1$. On the other hand, if $\alpha >1$ then $P[B(\varepsilon) \subset \cE] \le P[0 \in \cE] < 1$.
\end{proof}

\begin{remark}
Let $x, y \in \RR^2$ and $l(x,y)$ be the segment with endpoints in $x$ and $y$. Notice that $P[l(x,y) \subset \cE] \geq P[\exists z \in \supp \omega \cap B(x, 1/4); y \in E_z]$. Using Lemma \ref{lema:ellipse_covers_small_ball_step1} and the same thinning argument above we can see that $P[l(x,y) \subset \cE] \geq 1 - \exp[-u c |x - y|^{-\alpha}]$. Then, the ellipses model cannot be dominated by any Boolean model of fixed radius since for the Boolean model the probability of covering $l(x,y)$ decays exponentialy on $|x - y|$ (see Remark 3.2 of \cite{1010.5338}). 
\label{remark:not_dominated_by_Boolean}
\end{remark}

\subsection{Estimates for intersecting a ball}

In a completely analogous way we have just done, we can look into the case in which $\cE$ intersects a ball $B(w,a)$. 
\newconstant{const:step_1_for_A_ball}
\begin{lema}
\label{lema:prob_ellipse_intersects_ball_step1}
Let $C$ be any fixed constant with $C > 1$. Let $a > 0$ and $w, z \in \RR^2$ be points with $|z - w| \geq \max\{a+1, Ca\}$. Also, let $E_z$ be a random ellipse centered at $z$ and whose size and direction of the major axis have distribution $\rho \otimes \nu$, and let $\alpha$ be the decay parameter of $\rho$. Then, there exists $\uc{const:step_1_for_A_ball} = \uc{const:step_1_for_A_ball}(\uc{const:R_decay}, \alpha, C) > 0$ such that
\begin{equation*}
\uc{const:step_1_for_A_ball}^{-1} a |z - w|^{-(\alpha + 1)} \le P[E_z \cap B(w, a) \neq \varnothing] \le \uc{const:step_1_for_A_ball} (a + 1) |z - w|^{-(\alpha + 1)}.
\end{equation*}
\end{lema}

\begin{proof}
By applying a translation and a rotation we can suppose without loss of generality that $w$ is the origin and $z = \bigl( |z|, 0 \bigr)$. Thus, it is sufficient to prove there exists $\uc{const:step_1_for_A_ball}(\uc{const:R_decay}, \alpha, C) > 0$ such that 
\begin{equation*}
\uc{const:step_1_for_A_ball}^{-1} a |z|^{-(\alpha + 1)} \le P[E_z \cap B(a) \neq \varnothing] \le \uc{const:step_1_for_A_ball} (a + 1) |z|^{-(\alpha + 1)}.
\end{equation*}

\begin{figure}
\centering
\begin{tikzpicture}
 \clip (-1.3,-1.3) rectangle (8, 3.4);
 \draw[thick] (0,0) circle (1);
 \draw (-1.2,0) -- (8,0);
 
 \coordinate (z) at (7,0);
 \coordinate (w) at ({acos(2/7)}:2);
 \coordinate (unitary w) at   ({ acos(2/7)}:1);
 \coordinate (unitary z-w) at ({-asin(2/7)}:1);
 
 \filldraw[fill=gray!50,opacity=.4, rotate={180-asin(2/7)}] (z) circle (8.5 and 1);
 \fill (z) circle (.05) node[below] {$z$};
 
 \draw (0,0) -- (w) node[pos=.25,left] {$a$} node[pos=.75, left] {$1$};
 \draw (w)   -- (z);
 \draw[<->] ($ (z) - (1.8,0)$) arc (180:{180 - asin(2/7)}:1.8);
 \node at (4.8,.3) {$v_{a,z}$};
 
 \draw ($ (z)!-.2!(w) + (unitary w)$) -- ($ (z)!1.5!(w) + (unitary w)$);
 \draw ($ (z)!-.2!(w) - (unitary w)$) -- ($ (z)!1.5!(w) - (unitary w)$);
 
 \draw ($ (w) - .3*(unitary w)$) -- ++($ .3*(unitary z-w)$) -- ++($ .3*(unitary w)$);
\end{tikzpicture}
\caption{For $E_z \cap B(a) \neq \varnothing$, we need the restriction $|v(E_z)| \le v_{a,z}$.}
\end{figure}
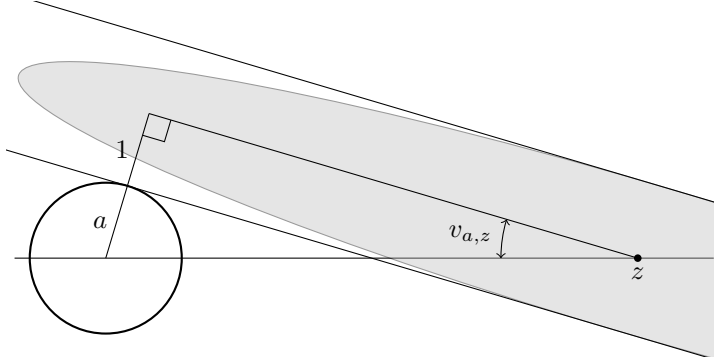

\vspace{2mm}
\noindent\textbf{Upper Bound:}
In order to be possible for the ellipse $E_z$ to intersect $B(a)$, it is necessary that $|V(E_z)| = |V| \le \arcsin (\tfrac{a+1}{|z|}) =: V_{a,z}$. Besides that, we also need the size of the major axis to be greater than a minimum value; notice that, independently of $V$, the value of $R(E_z)$ must be greater than $\bigl( |z| - a \bigr)$. Then, using independence
\begin{align*}
P[E_z \cap B(a) \neq \varnothing]
  &\le P\bigl[ \ R(E_z) \geq |z|-a\; , \ |V(E_z)| \le V_{a,z} \ \bigr] \le \uc{const:R_decay} \bigl( |z|-a \bigr)^{-\alpha} \cdot \left( \frac{2}{\pi} V_{a,z} \right) \\
  &= \frac{2}{\pi}\uc{const:R_decay} \left( 1 - \frac{a}{|z|} \right)^{-\alpha} |z|^{-\alpha} V_{a,z} \le \frac{2}{\pi}\uc{const:R_decay} \left( 1 - \frac{1}{C} \right)^{-\alpha} |z|^{-\alpha} V_{a,z}. 
\end{align*}

Finally, we use that $\arcsin x \le 2x$ for $x \in [0,1]$ to obtain $V_{a,z} = \arcsin ( \frac{a+1}{|z|} ) \le 2 (a + 1) |z|^{-1}$, because $a \geq 1$. Joining the last two equations, we have proven the upper bound with a constant $\uc{const:step_1_for_A_ball}(\uc{const:R_decay}, \alpha, C)$.

\vspace{3mm}
\noindent\textbf{Lower Bound:} Define $\tilde{V}_{a,z} := \arcsin( a / |z| )$. We claim that the event $\{ \, |V| < \tilde{V}_{a,z},\ R \geq \sqrt{|z|^2 - a^2}\, \}$ is contained in the event $\{E_z \cap B(a) \neq \varnothing\}$. Indeed, if $|V| < \tilde{V}_{a,z}$ then the direction of the major axis of $E_z$ must intersect $B(a)$. Requiring also that $R \geq \sqrt{|z|^2 - a^2}$ ensures the intersection. We carry out the same calculations as in the upper bound case:
\begin{align*}
P[E_z \cap B(a) \neq \varnothing]
  &\geq  P\left[\ |V| < \tilde{V}_{a,z}, \ R \geq \sqrt{|z|^2 - a^2}\ \right] \stackrel{\text{(Indep.)}}{\geq}  \frac{2}{\pi} \tilde{V}_{a,z} \cdot \uc{const:R_decay}^{-1} \bigl( |z|^2 - a^2 \bigr)^{-\frac{\alpha}{2}} \\
  &\geq \frac{2}{\pi} \tilde{V}_{a,z} \uc{const:R_decay}^{-1} \left( 1 - \frac{1}{C^2} \right)^{-\frac{\alpha}{2}} |z|^{-\alpha}  =    \uc{const:step_1_for_A_ball}(\uc{const:R_decay}, \alpha, C)^{-1} \arcsin\left(\frac{a}{|z|}\right) |z|^{-\alpha}
\end{align*}
for a constant $\uc{const:step_1_for_A_ball}$ possibly greater than the one obtained in the upper bound. The lemma is proven, since $\arcsin (a/|z|) \geq a/|z|$.
\end{proof}

From Lemma \ref{lema:prob_ellipse_intersects_ball_step1} we can deduce the asymptotic behavior of $P[B(w,a) \cap \cE \neq \varnothing]$. Obviously, Proposition \ref{prop:prob_small_ball_covered_by_ellipse_step2} shows that when $\alpha \le 1$ this probability must be one, independently of $a$; so we can restrict ourselves to case where $\alpha > 1$.

\newconstant{const:lemma_ellipse_touches_ball_from_far}
\newconstant{const:new_ellipse_touches_ball}
\begin{prop}
\label{prop:new_ellipse_touches_ball}
Let $\rho$ have decay $\alpha > 1$ and fix some $a \geq 1$. It holds:
\begin{enumerate}[(i)]
\item Let $w \in \RR^2$ and $C \geq 2$ and define $g = g_{a,w,C}$ by $g(z) := P[ E_z \cap B(w, a) \neq \varnothing] \inds{B(w, Ca)^c}(z)$, which means that the thinning will keep only ellipses centered outside $B(w,Ca)$ that intersect $B(w,a)$. Then there is a positive constant $\uc{const:lemma_ellipse_touches_ball_from_far} = \uc{const:lemma_ellipse_touches_ball_from_far}(\uc{const:R_decay}, \alpha, C)$ such that 
\begin{equation*}
\exp[ - u \uc{const:lemma_ellipse_touches_ball_from_far}^{-1} a^{2 - \alpha}  ] \le
P[\omega_g(\RR^2) = 0] \le \exp[ - u \uc{const:lemma_ellipse_touches_ball_from_far} a^{2 - \alpha}  ]. 
\end{equation*}

\item There is a positive constant $\uc{const:new_ellipse_touches_ball} = \uc{const:new_ellipse_touches_ball}(\uc{const:R_decay}, \alpha)$ such that 
\begin{equation*}
1 - \exp[ - u \uc{const:new_ellipse_touches_ball}^{-1} a^{2}  ] \le P[B(w,a) \cap \cE \neq \varnothing]  \le 1 - \exp[ - u \uc{const:new_ellipse_touches_ball} a^2  ].
\end{equation*}
\end{enumerate}
\end{prop}

\begin{remark}
\label{remark:ellipse_touches_ball_from_far}
Proposition \ref{prop:new_ellipse_touches_ball}.\textit{(i)} is indeed quite useful. It allows us to disregard the influence of ellipses too far away from the region we are interested in. For instance, when $\alpha > 2$ and $a$ is sufficiently large, the probability of $\{\omega_g(\RR^2) = 0 \}$ is close to 1. This means that we can pay a small price for assuming that ellipses far away (centered on $B(Ca)^c$) do not interfere in what happens on the ball $B(a)$. The case in which $\alpha = 2$ is of special interest, as we will see. 
\end{remark}

\begin{proof}
Once again, we can assume without loss of generality that $w = 0$. Part \textit{(i)} is a straightforward application of Proposition \ref{prop:thinning} and Lemma \ref{lema:prob_ellipse_intersects_ball_step1}. The conditions $a \geq 1$ and $C \geq 2$ are just simple requirements to force $Ca \geq a + 1$, so that $g(z) \asymp |z|^{-(\alpha + 1)}$ for $|z| > Ca$.

For part $\textit{(ii)}$, notice that if we take $g$ as in part \textit{(i)} with $C = 2$ then
\begin{equation*}
P[\omega(B(2a)) = 0, \omega_g(B(2a)^c) = 0] \le P[B(a) \subset \cV] \le P[\omega(B(a)) = 0].
\end{equation*}

Since the random variables $\omega(B(2a))$ and $\omega_g(B(2a)^c)$ are independent and $P[\omega(B(2a)) = 0] = \exp[- u \pi 4 a^2]$, we conclude from part \textit{(i)} that 
\begin{equation}
\exp[ - u 4\pi a^2 - u \uc{const:lemma_ellipse_touches_ball_from_far}^{-1} a^{2 - \alpha}  ] \le
P[B(a) \subset \cV] \le \exp[ - u \pi a^{2}  ].
\end{equation}
The result follows after we notice that $a^2 \geq a^{2 - \alpha}$ for $a \geq 1$ and define $\uc{const:new_ellipse_touches_ball}(\uc{const:R_decay}, \alpha)$ appropriately.
\end{proof}

\section{Phase Transition for Total Covering}
\label{section:phase_transition_for_total_covering}

The proof of Theorems \ref{teo:alpha_phase_transition} and \ref{teo:covered_set_alpha_phase_transition} is split into two sections. In this section, we answer the question: are there values of $\alpha$ and $u$ for which the vacant set $\cV$ is empty almost surely? In other words, such that the plane is completely covered by ellipses, with probability one? Observe that as the value of $\alpha$ decreases, ellipses with very large axes will become more frequent. As a consequence, the region covered by the ellipses tends to be greater.

Proposition \ref{prop:prob_small_ball_covered_by_ellipse_step2} is sufficient to provide an answer of how total covering depends on $\alpha$.

\begin{prop}
\label{prop:total_covering}
We have $P[\cE = \RR^2] = 1$ if and only if $\alpha \le 1$.
\end{prop}

\begin{proof}[Proof of Proposition \ref{prop:total_covering}]
Suppose $\alpha \le 1$. It follows from Proposition \ref{prop:prob_small_ball_covered_by_ellipse_step2} that $P(B(w, 1/4) \subset \cE, \ \forall w \in \QQ^2) =1$. Since this event is the same as $\{\cE = \RR^2\}$ we have proven one of the implications. To see the other, it suffices to see that when $\alpha > 1$ Proposition \ref{prop:prob_small_ball_covered_by_ellipse_step2} says that $P( 0 \in \cE ) < 1$.
\end{proof}

\subsection{Infinite Area Argument}
\label{subsection:infinite_area_argument}

Recall that $\lambda$ is the Lebesgue measure on $\RR^2$. Noticing that $\lambda(E_0) = \pi R(E_0)$ has infinite expected value if and only if $ \alpha \le 1$, Proposition \ref{prop:total_covering} can be restated in the suggestive form:

\begin{coro}
\label{coro:total_covering_infinite_area}
We have $P[\cE = \RR^2] = 1$ if and only if $E[\lambda(E_0)] = \infty$.
\end{coro}

Corollary \ref{coro:total_covering_infinite_area} says the probability of total covering is related to the expected value of the area of the random subsets we are working with. This fact is not a coincidence and can be used to extend the proof of total covering to a more general setting. Consider a model of random subsets made by taking $\omega$ a PPP($u\lambda$) on $\RR^2$ and associating to every $z \in \supp \omega$ a random closed subset $E_z \subset \RR^2$ independently of everything else (see \cite{hall1985continuum}). We take 
\begin{equation*}
\cE = \bigcup_{z \in \supp \omega} \hspace{-2mm} E_z.
\end{equation*}

In \cite{hall1985continuum}, it is proven that for any bounded measurable set $A \subset \RR^2$ we have $P[\lambda(A \backslash \cE) = 0] = 1$ if and only if $E[\lambda(E_0)] = \infty$, which implies that $\lambda(\cV) = 0$ almost surely. In general, this does not mean that $\cV = \varnothing$ a.s.. However, we can use this fact to prove total covering for any ellipses model.

\begin{proof}[Proof of Corollary \ref{coro:total_covering_infinite_area} for ellipses model]
Fix any ellipses model and choose $\varepsilon < 1/2$. Denote by $l(E_0)$ the perimeter of ellipse $E_0$ and notice that if an ellipse has axes of size $a$ and $b$ then its perimeter $p$ satisfies $p \le \pi \sqrt{2(a^2 + b^2)}$ (see eg. \cite{Klamkin1971}). In our case,
\begin{equation}
l(E_0) \le \pi \sqrt{2(1 + R^2)} \le 2\pi R \le 2\lambda(E_0).
\label{eq:relation_area_perimeter_ellipse}
\end{equation}

For a set $K \subset \RR^2$, denote by $K^{-\varepsilon}$ the $\varepsilon$-interior of the set $K$, that is $K^{-\varepsilon} = \{x; B(x,\varepsilon) \subset K\}$. Consider the model where the random subsets are given by the family $(E_z^{-\varepsilon})_{z \in \supp \omega}$ and denote its vacant set by $\tilde{\cV}$. Notice that the models we are considering are supported on bounded convex subsets. If $K$ is a bounded convex subset, as a particular case of Steiner-Minkowski formula (see eg. \cite{berger1987geometry}) we have $\lambda(K + B(r)) = \lambda(K) + \lambda(B(r)) + r \, l(K)$. Applying to $K = E_0^{-\varepsilon}$ and $r = \varepsilon$, we obtain:
\begin{align}
\lambda(E_0) 
  &= \lambda(E_0^{-\varepsilon} + B(\varepsilon)) = \lambda(E_0^{-\varepsilon}) + \lambda(B(\varepsilon)) + \varepsilon \, l(E_0^{-\varepsilon}) \nonumber\\
  &\le \lambda(E_0^{-\varepsilon}) + \lambda(B(\varepsilon)) + \varepsilon \, l(E_0) \le \lambda(E_0^{-\varepsilon}) + \lambda(B(\varepsilon)) + \varepsilon 2 \, \lambda(E_0)
\end{align}
where the first inequality comes from the fact that if $K_1 \subset K_2$ are two bounded convex subsets of $\RR^2$ then $l(K_1) \le l(K_2)$ and the second comes from \eqref{eq:relation_area_perimeter_ellipse}. Then, it follows $(1 - 2\varepsilon)\lambda(E_0) - \lambda(B(\varepsilon)) \le \lambda(E_0^{-\varepsilon}) \le \lambda(E_0)$ and we conclude $E[\lambda(E_0)] = \infty$ if and only if $E[\lambda(E_0^{-\varepsilon})] = \infty$. By this, we have $\lambda(\tilde{\cV}) = 0$ a.s. and thus $\cV = \varnothing$.
\end{proof}

\begin{remark}
The proof above can be immediately generalized for any model supported on bounded convex sets of $\RR^2$ satisfying for some universal constant $C > 0$ the relation $l(E_0) \le C \lambda(E_0)$. Although there are many papers on the total covering of sets, especially in relation to Dvoretsky's covering problem, we did not find any result that would apply to ellipses model directly (see eg. Kahane \cite{kahane2000random}).
\end{remark}

\subsection{Quantitative Estimates}

Section \ref{subsection:infinite_area_argument} has a proof of the phase transition for total covering in the ellipses model that does not need any of the estimates of Section \ref{section:probability_of_simple_events}. This fact might put into question whether those estimates are useful at all. Actually, such estimates will have greater importance in the subsequent sections.

In this section we emphasize that the bounds from Section \ref{section:probability_of_simple_events} provide a more precise description of the ellipses model. To exemplify that, we will prove a stronger result about the covering of a small ball by ellipses, generalizing Proposition \ref{prop:prob_small_ball_covered_by_ellipse_step2}.

For $\varepsilon > 0$, consider the random variables
\begin{equation*}
N^{(\varepsilon)}_n := \#\{s \in \supp \xi;\ B(\varepsilon) \subset E(s), \ c(E(s)) \in B(n)\},
\end{equation*}
the number of ellipses centered on the euclidean ball $B(n)$ that cover the ball $B(\varepsilon)$. For a fixed $\varepsilon < 1/2$, we prove a law of large numbers for $N^{(\varepsilon)}_n$.

\begin{prop}
Let $\varepsilon < 1/2$. We have that:

\begin{enumerate}[1)]
\item For $0 < \alpha < 1$, it holds $E[N^{(\varepsilon)}_n] \asymp n^{1-\alpha}$ and $\frac{N^{(\varepsilon)}_n - E[N^{(\varepsilon)}_n]}{\sqrt{n^{1-\alpha} (\log n)^{1 + \delta}}} \to 0$ as. when $n \to \infty$, for fixed $\delta > 0$.

\item For $\alpha = 1$, it holds $E[N^{(\varepsilon)}_n] \asymp \log n$ and $\frac{N^{(\varepsilon)}_n - E[N^{(\varepsilon)}_n]}{n^{1/2} (\log n)^{1 + \delta}} \to 0$ as. when $n \to \infty$, for fixed $\delta > 0$.
\end{enumerate}
\end{prop}

\begin{proof}
We omit the details. Notice that the random variables $N^{(\varepsilon)}_n$ have distribution 
\begin{equation*}
N^{(\varepsilon)}_n \distr \Poi\left(u \int_{B(n)} \hspace{-4mm}P[B(\varepsilon) \subset E_z] \, \text{d}z\right).
\end{equation*}

Then, the asymptotic estimates for $E[N^{(\varepsilon)}_n]$ follow from Lemma \ref{lema:ellipse_covers_small_ball_step1}. Define the random variables $X_n^{(\varepsilon)} := N^{(\varepsilon)}_n - N^{(\varepsilon)}_{n-1}$, which are independent Poisson random variables with 
\begin{equation*}
X_n^{(\varepsilon)} \distr \Poi\left( u \int_{B(n)\backslash B(n-1)} \hspace{-15mm}P[B(\varepsilon) \subset E_z] \, \text{d}z \right).
\end{equation*}

In order to prove a strong law of large numbers, we resort to a theorem of Kolmogorov (see \cite{shiryaev1996probability}, Theorem 2 on page 389). Applied to $(X_n^{(\varepsilon)})$, it states that for any sequence of numbers $(b_n) \subset \RR^+$ with $b_n \uparrow \infty$ and $\sum \frac{\Var X^{(\varepsilon)}_n}{b_n^2} < \infty$ we have 
\begin{equation*}
\frac{N^{(\varepsilon)}_n - E[N^{(\varepsilon)}_n]}{b_n} = \frac{\sum_{j=1}^{n} X^{(\varepsilon)}_j - \sum_{j=1}^{n} E[X^{(\varepsilon)}_j]}{b_n} \to 0 \ \text{as. when $n \to \infty$.}
\end{equation*}

To finish the proof, we notice that for Poisson random variables the expectation and the variance coincide. The sequences $b_n$ were chosen to use the fact that $\sum_n \tfrac{1}{n (\log n)^{q}} < \infty$ if and only if $q > 1$.
\end{proof}

\section{Phase Transition for Existence of Critical Point}
\label{section:phase_transition_for_existence_of_critical_point}

Let us define two critical values for $u$ in the $(u, \rho)$-ellipses model:

\begin{defi}
Define the critical values
\begin{equation*}
\bar{u}_c(\rho) := \inf \{ u \geq 0 ; \ P_{u, \rho}[\cV \ \text{percolates}] = 0 \} \ \text{and} \ u_c(\rho) := \inf \{ u \geq 0 ; \ P_{u, \rho}[\cE \ \text{percolates}] = 1 \}.
\end{equation*}
\end{defi}

\begin{remark}
\label{remark:properties_of_uc_and_bar_uc}
We make some comments about how $u_c(\rho)$ and $\bar{u}_c(\rho)$ are related:
\begin{enumerate}
\item Recall that we assumed $\rho$ is supported on $[1, \infty)$. By this, our model trivially dominates Poisson Boolean percolation with circles of radius 1 for any $\rho$. Using this fact, it is easy to prove percolation for the covered set $\cE$ when $u$ is sufficiently large. Thus, there exists a finite constant $C$ such that $u_c(\rho) \le C$ and $\bar{u}_c(\rho) \le C$ for all $\rho$ we are considering. Moreover, notice that $u_c(\rho)$ may assume different values even for $\rho$'s with the same decay $\alpha$; the same goes for $\bar{u}_c(\rho)$.

\item One could try to adapt the classical proof of uniqueness of the infinite cluster in the supercritical phase to the covered set, together with Zhang's argument to conclude that in any ellipses model infinite vacant and covered clusters cannot coexist. We believe this holds, but did not carry out the computations. If true, this would imply $\bar{u}_c(\rho) \le u_c(\rho)$.

\item Notice that the critical values do not need to be equal, since we prove with Theorem\ref{teo:vacant_crossing_of_boxes_alpha_2} that $\bar{u}_c(\rho) = 0 < u_c(\rho)$ when $\alpha = 2$. However, as we stated in Conjecture \ref{conj:equality_critical_points_alpha_greater_2}, when $\alpha > 2$ we believe equality actually holds.
\end{enumerate}
\end{remark}

In the previous section we already proved that $\cV \neq \varnothing$ for $\alpha > 1$. Now we deal with the second phase transition. In this section we finish the proof of Theorems \ref{teo:alpha_phase_transition} and \ref{teo:covered_set_alpha_phase_transition}.

\subsection{Crossing a box with one ellipse}
\label{subsection:crossing_box_with_one_ellipse}

Let us estimate the probability of the event that a single ellipse manages to connect opposite sides of a fixed box. This subsection could be at Section \ref{section:probability_of_simple_events}, but we chose to put it closer to where it is needed. Proposition \ref{prop:crossing_box_with_one_ellipse} below will be useful for proving Theorem \ref{teo:vacant_crossing_of_boxes_alpha_2} also.

Recall our notation for boxes $B_{\infty}(l;k)$ and its sides $L^{-}(l;k)$ and $L^{+}(l;k)$. Also, recall that for any ellipse $E$ we defined $c(E)$ as its center, $R(E)$ as the size of its major axis and $V(E)$ as the direction of its major axis.

\begin{defi}
Define the events
$$
 LR(l;k) := \left\{ \exists \gamma: [0,1] \to \RR^2; 
     \begin{array}{l}
          \text{$\gamma$ is continuous,} \ \gamma([0,1]) \subset \cE \cap B_\infty(l;k), \\
          \gamma(0) \in L^-(l;k) \ \text{and} \ \gamma(1) \in L^+(l;k)
     \end{array} 
 \right\}
$$

$$
 LR_1(l;k) := \left\{ \exists s \in \supp \xi ; 
     \begin{array}{l}
          E(s) \cap L^{-}(l;k) \neq \varnothing \ \text{\upshape  and } \\
          E(s) \cap L^{+}(l;k) \neq \varnothing
     \end{array}
 \right\}
$$
\end{defi}

In words, $LR(l;k)$ denotes the event in which there is a left-right crossing of $B_\infty(l;k)$ contained on $\cE$ and $LR_1(l;k)$ is the event in which such a crossing is obtained by one ellipse alone. The subscript $1$ in the above notation is to emphasize this. Obviously, $LR_1(l;k) \subset LR(l;k)$.

Let us prove bounds for $P[LR_1(l;k)]$. Firstly, we handle the easiest case. In Section \ref{section:phase_transition_for_total_covering} we proved that $P[\cE = \RR^2] = 1$ when $\alpha \le 1$. Therefore $P[LR_1(l;k)] = 1$ for these values of $\alpha$ and we omit the proof of this result. We only have to be concerned with the case where $\alpha > 1$. In this case, it holds

\newconstant{const:crossing_box_with_one_ellipse}
\begin{prop}
\label{prop:crossing_box_with_one_ellipse}
If $\alpha >1$ and $k, l > 0 $ satisfy $lk > 2$, then there is a constant $\uc{const:crossing_box_with_one_ellipse} = \uc{const:crossing_box_with_one_ellipse}(\alpha, \uc{const:R_decay}) > 0$ such that:
\begin{equation}
1 - \exp[ - \uc{const:crossing_box_with_one_ellipse}^{-1} u (k \wedge k^{-\alpha}) l^{2 - \alpha}] \le P(LR_1(l;k)) \le 1 - \exp[ - \uc{const:crossing_box_with_one_ellipse} u (k^{2-\alpha}\vee k^{-\alpha}) l^{2 - \alpha}].
\label{eq:crossing_box_with_one_ellipse}
\end{equation}
\end{prop}

\begin{remark}
The restriction $lk > 2$ is necessary to avoid that the horizontal length of $B_\infty(l;k)$, given by $kl$, were too small. In that case, a well positioned center of ellipse guarantees the crossing, independently of its major axis direction and size. Since we are mainly interested in cases in which $kl \to \infty$, this restriction is harmless.
\end{remark}

\begin{remark}
The lower bound will be important for Section \ref{subsection:triviality_of_critical_points_when_alpha_in_1_2}. Also, notice that when $\alpha = 2$ and $k$ is fixed then $P[LR_1(l;k)]$ is bounded away from $0$ and $1$ uniformly on $l$. This property plays an important role in Section \ref{section:vacant_crossing_of_boxes_for_alpha_equal_2}.
\end{remark}

\begin{proof}
We begin proving the lower bound.

\vspace{2mm}
\noindent\textbf{Lower Bound:} In order to find a lower bound for $P[LR_1(l;k)]$ let us study an event contained in the event $LR_1(l;k)$. Instead of searching all $\RR^2$ for some ellipse that makes the crossing, we can restrict our search to a region that is simpler to analyze. We may force the center of the ellipse to be in the interior of $B_{\infty}(l;k)$. Using the notation above, define the event
\begin{equation*}
LR_1^{-}(l;k) := \left\{ \exists s \in \supp \xi ; 
     \begin{array}{l}
          E(s) \cap L^{-}(l;k) \neq \varnothing, E(s) \cap L^{+}(l;k) \neq \varnothing \\
          \text{and} \ c\bigl( E(s) \bigr) \in B_{\infty}(l/2; k)
     \end{array}
 \right\}.
\end{equation*}

We want that at least one ellipse intersects both $L^{-}(l;k)$ and $L^{+}(l;k)$. If we fix the center of the ellipse, this implies a lower bound for $R(E)$, the size of its major axis. However, it is not enough that $R(E)$ is sufficiently large. It is also necessary to consider the direction of its major axis $V(E)$. The choice of restricting to a subregion of $B_\infty(l;k)$ comes in handy now; independently of where $c(E)$ is, if we know that $c(E) \in B_{\infty}(l/2;k)$ then 
\begin{equation}
\label{eq:suff_cond_for_ellipse_in_LRminus}
\left\{
\begin{array}{l}
          V(E) \in \left( - \arctan\left( \frac{1}{3k}\right), \arctan\left( \frac{1}{3k}\right)\right) \\
           \text{ and } R(E) \geq \frac{l}{4} \sqrt{1 + 9k^2}
     \end{array}
\right\}
\ \text{implies} \  
\left\{
     \begin{array}{l}
          E \cap L^{-}(l;k) \neq \varnothing \text{ and } \\
          E \cap L^{+}(l;k) \neq \varnothing
     \end{array}
\right\},
\end{equation}

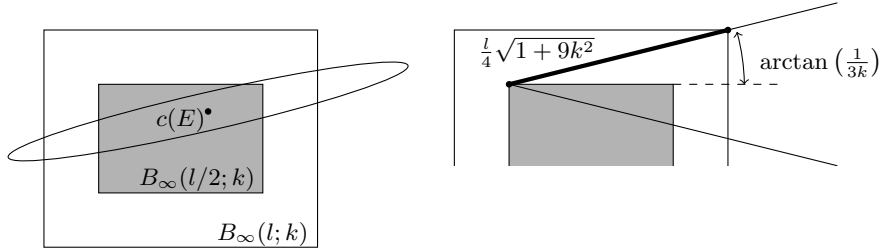
\begin{figure}
\centering
\begin{tikzpicture}[scale=.36]
 \draw (0,0) rectangle (10,8);
 \draw[fill=gray!60] (2,2) rectangle (8,6);
 \node at (8,.5) {\small $B_\infty(l;k)$};
 \node at (5.6,2.5) {\small $B_\infty(l/2; k)$};
 \begin{scope}[shift={(-1,0)}]
 \filldraw (7,5) circle (.1);
 \node at (6,4.7) {\small $c(E)$};
 \draw[rotate around={13:(7,5)}] (7,5) circle (7.5 and .7);
 \end{scope}

 \begin{scope}[shift={(15,-2)}]
 \clip (-0.3,5) rectangle (17,12);
 \draw (0,0) rectangle (10,10);
 \draw[fill=gray!60] (2,2) rectangle (8,8);
 \draw[dashed] (8,8) -- (12,8);
 \draw         (2,8) -- (14,11);
 \draw         (2,8) -- (14, 5);
 \filldraw (2,8)   circle (.1);
 \filldraw (10,10) circle (.1);
 \draw[ultra thick]  (2,8) -- (10,10);
 \draw[<->] (10.6,8) arc (1:15:8);
 \node at (13.4,8.8) {\small $\arctan\left( \frac{1}{3k}\right)$};
 \node at (3,9.2) {\small $\frac{l}{4} \sqrt{1 + 9k^2}$};
 \end{scope}
\end{tikzpicture}
\caption{Condition in \eqref{eq:suff_cond_for_ellipse_in_LRminus} implies the event $LR_1^{-}(l;k)$ happens.}
\label{fig:condition_for_LR1_minus}
\end{figure}

\noindent as is represented in Figure \ref{fig:condition_for_LR1_minus}. Denote $V_{\max} := \arctan\left( \frac{1}{3k}\right)$ and $R_{\min} := \frac{l}{4}\sqrt{1 + 9k^2}$. For $z \in B_{\infty}(l/2; k)$, if we choose randomly an ellipse $E_z$ centered in $z$ with major axis distributed like $\rho\otimes\nu$ as usual, we have
\begin{align}
P\left[ 
\begin{array}{l}
          E_z \cap L^{-}(l;k) \neq \varnothing, \\
          E_z \cap L^{+}(l;k) \neq \varnothing
     \end{array}
\right]
 &\overset{\phantom{\text{Ind.}}}{\geq}
P\left[ 
\begin{array}{l}
          R(E_z) \geq R_{\min}, \\
          |V(E_z)| \le V_{\max}
     \end{array}
\right] \nonumber\\
 &\overset{\text{Ind.}}{=}  P( |V(E_z)| \le V_{\max}) \, P( R(E_z) \geq R_{\min} ) \geq 2V_{\max} \cdot \uc{const:R_decay}^{-1} R_{\min}^{-\alpha}. \label{eq:low_bound_thin_LR1_minus}
\end{align}

Notice that in the last computation we needed to verify that $R_{\min} \geq 1$, otherwise $P( R(E_z) \geq R_{\min} ) = 1$. However, our hypothesis that $kl > 2$ ensures $R_{\min} \geq 1$. To achieve the wished lower bound, we turn to Proposition \ref{prop:thinning} with the function 
\begin{equation*}
g(z) = P[E_z \cap L^{-}(l;k) \neq \varnothing,\; E_z \cap L^{+}(l;k) \neq \varnothing] \, \inds{B_\infty(l/2;k)}(z),
\end{equation*}
which by \eqref{eq:low_bound_thin_LR1_minus} satisfies $g(z) \geq 2V_{\max} \cdot \uc{const:R_decay}^{-1} R_{\min}^{-\alpha}\; \inds{B_\infty(l/2;k)}(z)$. Thus, we have
\begin{align}
P(LR_1(l;k)) &\geq P[LR_1^-(l;k)] \geq P[\omega_g(\RR^2) \geq 1] = 1 - \exp\left[ - u \int_{\RR^2} g(z) \;\text{d}z \right] \nonumber\\
             &\geq 1 - \exp\left[ - u \int_{B_\infty(l/2;k)} 2V_{\max} \cdot \uc{const:R_decay}^{-1} R_{\min}^{-\alpha} \;\text{d}z \right] \nonumber\\
             &=    1 - \exp\left[ - u \cdot 2\arctan\left( \tfrac{1}{3k} \right)  \cdot \uc{const:R_decay}^{-1} l^{-\alpha} 4^{\alpha}(1 + 9k^2)^{-\alpha/2} \cdot \frac{kl^2}{4}\right] \nonumber\\
             &\geq 1 - \exp\left[ - u\; \uc{const:crossing_box_with_one_ellipse}^{-1} \cdot \arctan\left( \tfrac{1}{3k} \right)k (k + 1)^{-\alpha} \cdot l^{2-\alpha}\right]
\label{eq:low_bound_LR1(lk)_part1}
\end{align}
for some constant $\uc{const:crossing_box_with_one_ellipse} = \uc{const:crossing_box_with_one_ellipse}(\uc{const:R_decay}, \alpha)$.

To simplify the function $f(k) = \arctan\left( \tfrac{1}{3k} \right)k (k + 1)^{-\alpha}$, that appears on equation \eqref{eq:low_bound_LR1(lk)_part1}, we notice that $f(k) \sim \tfrac{\pi}{2} k$ when $k \to 0$, $f(k) \sim \tfrac{1}{3}k^{-\alpha}$ when $k \to \infty$ and $f$ is a continuous, positive function on $(0, \infty)$. By this, changing the constant $\uc{const:crossing_box_with_one_ellipse}$ if needed we can assure that:
\begin{equation}
P(LR_1(l;k)) \geq P(LR_1^{-}(l;k)) \geq 1 - \exp[ - u \uc{const:crossing_box_with_one_ellipse}^{-1} (k \wedge k^{-\alpha}) l^{2 - \alpha}]. 
\label{eq:low_bound_LR1(lk)}
\end{equation}

\vspace{3mm}
\noindent\textbf{Upper Bound:} To prove the upper bound, we decompose the event $LR_1(l;k)$ into two independent events. The idea is to decompose it with respect to the position of the ellipse that makes the crossing. To simplify the notation, we denote $a = (k \vee 1)l$. With this, notice that $B_{\infty}(l;k) \subset B(a) \subset B(2a)$. Define the events
\begin{equation*}
LR_1^{1}(l;k) := \left\{ \exists s \in \supp \xi ; 
     \begin{array}{l}
          E(s) \cap L^{-}(l;k) \neq \varnothing, E(s) \cap L^{+}(l;k) \neq \varnothing \\
          \text{and} \ c(E(s)) \in B(2a)
     \end{array}
 \right\},
\end{equation*}
\begin{equation*}
LR_1^{2}(l;k) := \left\{ \exists s \in \supp \xi ; 
     \begin{array}{l}
          E(s) \cap L^{-}(l;k) \neq \varnothing, E(s) \cap L^{+}(l;k) \neq \varnothing \\
          \text{and} \ c(E(s)) \notin B(2a)
     \end{array}
 \right\}.
\end{equation*}

Omitting the dependence on $l$ and $k$ and observing that the above defined events are independent, it holds that
\begin{equation}
\label{eq:decomposition_LR1_into_LR1j}
P[LR_1] = P[LR_1^1 \cup LR_1^2] = 1 - P[(LR_1^1)^c \cap (LR_1^2)^c] = 1 - P[(LR_1^1)^c]P[(LR_1^2)^c].
\end{equation}

\vspace{3mm}
\noindent\textbf{Bound for $LR_1^2$:} Initially, by event inclusion, notice that the following inequality holds
\begin{equation*}
P[LR_1^2(l;k)] \le P[\exists s \in \supp \xi; \ c(E(s)) \notin B(2a) \text{ and } E(s) \cap B(a) \neq \varnothing].
\end{equation*}

Apply Proposition \ref{prop:new_ellipse_touches_ball}.\textit{(i)} with $C = 2$ and $g(z) = P[ E_z \cap B(a) \neq \varnothing] \inds{B(2a)^c}(z)$ to deduce
\begin{equation}
\label{eq:up_bound_for_LR2_1}
P[LR_1^2(l;k)^c] \geq P[\omega_g(\RR^2) = 0] \geq \exp[ - u \uc{const:lemma_ellipse_touches_ball_from_far}^{-1} a^{2 - \alpha}  ]
\end{equation}
and then define a constant $\hat{\uc{const:crossing_box_with_one_ellipse}} = \hat{\uc{const:crossing_box_with_one_ellipse}}(\uc{const:R_decay}, \alpha)$ with the same value as $\uc{const:lemma_ellipse_touches_ball_from_far}^{-1}$.

\vspace{3mm}
\noindent\textbf{Bound for $LR_1^1$:} Notice that for any ellipse $E$, the farthest point covered by $E$ from its center $c(E)$ is at distance $R(E)$. Then, for an ellipse $E(s)$ with $s \in \supp \xi$ to be able to connect both sides of $B_{\infty}(l;k)$, it is necessary that
\begin{equation*}
R(E(s)) \geq \max \{ \dist(c(E(s)), L^-(l;k)), \dist(c(E(s)), L^+(l;k)) \} \geq \frac{lk}{2}.
\end{equation*}

In this way, for any center of ellipse $z$ that is inside the ball $B(2a)$ we have
\begin{equation}
P[E_z \cap L^-(l;k) \neq \varnothing, \ E_z \cap L^+(l;k) \neq \varnothing] \le P[R(E_z) \geq lk/2] \le \uc{const:R_decay}(lk/2)^{-\alpha}.
\label{eq:up_bound_for_LR1_1}
\end{equation}

Now, we want to apply Proposition \ref{prop:thinning} to the function 
\begin{equation*}
g(z) = P[E_z \cap L^-(l;k) \neq \varnothing, \ E_z \cap L^+(l;k) \neq \varnothing] \;\inds{B(2a)}(z).
\end{equation*}

Notice that $g(z) \le \uc{const:R_decay}2^\alpha (lk)^{-\alpha} \, \inds{B(2a)}(z)$ by equation \eqref{eq:up_bound_for_LR1_1} and hence
\begin{align}
P[LR_1^1(l;k)^c] 
  &\geq P[\omega_g(\RR^2) =0] = \exp\left[ - u \int_{\RR^2} g(z) \;\text{d}z \right] \geq \exp\left[ - u \int_{B(2a)} \hspace{-4mm}\uc{const:R_decay}2^\alpha (lk)^{-\alpha} \;\text{d}z \right] \nonumber\\
  &= \exp\left[ - u \uc{const:R_decay} 2^\alpha (lk)^{-\alpha} \pi(2a)^2 \right] = \exp\left[ - u \tilde{\uc{const:crossing_box_with_one_ellipse}} k^{-\alpha}(k\vee 1)^2 l^{2 - \alpha}\right],
\label{eq:up_bound_for_LR1_2}
\end{align}
where we defined a constant $\tilde{\uc{const:crossing_box_with_one_ellipse}} = \tilde{\uc{const:crossing_box_with_one_ellipse}}(\uc{const:R_decay}, \alpha)$. To finish the upper bound for the probability of $LR_1(l;k)$, we substitute on equation \eqref{eq:decomposition_LR1_into_LR1j} the 
estimates from equations \eqref{eq:up_bound_for_LR2_1} and \eqref{eq:up_bound_for_LR1_2}:
\begin{align*}
P[LR_1(l;k)]
 &\le  1 - \exp \left[ - u \tilde{\uc{const:crossing_box_with_one_ellipse}} k^{-\alpha}(k\vee 1)^2 l^{2-\alpha} \right] \exp[ - u \hat{\uc{const:crossing_box_with_one_ellipse}} (k\vee 1)^{2 - \alpha} l^{2-\alpha}] \\
 &\le  1 - \exp[ - u \uc{const:crossing_box_with_one_ellipse} k^{-\alpha}(k\vee 1)^2 l^{2-\alpha}] = 1 - \exp[ - u \uc{const:crossing_box_with_one_ellipse} (k^{2-\alpha}\vee k^{-\alpha}) l^{2-\alpha}].
\end{align*}

Here, we used that $k^{-\alpha}(k\vee 1)^2 \geq (k\vee 1)^{2 - \alpha}$ and took $\uc{const:crossing_box_with_one_ellipse}(\uc{const:R_decay}, \alpha) = \hat{\uc{const:crossing_box_with_one_ellipse}} + \tilde{\uc{const:crossing_box_with_one_ellipse}}$.
\end{proof}

\subsection{Triviality of Critical Points when \texorpdfstring{$\alpha \in (1,2)$}{alpha is in (1,2)}}
\label{subsection:triviality_of_critical_points_when_alpha_in_1_2}

We use Proposition \ref{prop:crossing_box_with_one_ellipse} and Borel-Cantelli's lemma. For the covered set our proof is straightforward.

\begin{proof}[Proof of Theorem \ref{teo:covered_set_alpha_phase_transition}.2]
Consider the boxes $B_{n} = [0,2^{n+1}]\times [0,2^n]$ for $n$ odd and $B_n = [0,2^n]\times [0,2^{n+1}]$ for $n$ even. If for all sufficiently large $n$ we have horizontal crossings of $B_n$ for $n$ odd and vertical crossings of $B_n$ for $n$ even, then it is clear that $\cE$ percolates. Proposition \ref{prop:crossing_box_with_one_ellipse} proves that
\begin{equation*}
P[\{\text{crossing of $B_n$ by one ellipse}\}^c] \le \exp[-u c (2^{2- \alpha})^n],
\end{equation*}
which is summable, so Borel-Cantelli implies $P[\liminf_n \{\text{crossing of $B_n$ by one ellipse}\}] = 1$.
\end{proof}

Now, we prove that almost surely $\cV$ does not percolate for any $u > 0$ when $\alpha \in (1,2)$. One way to prove percolation does not happen is to use an argument of duality. We would like to prove that, almost surely, there is an infinite collection of circuits of covered areas around the origin.

Working with ellipses that can be centered anywhere on $\RR^2$, it can be tricky to analyze general circuits of ellipses. In our proof, we replicate the idea in \cite{1010.5338}, Proposition 5.6. We show that it is enough to look at a very special kind of collection of circuits; every circuit in the collection will be made of three carefully positioned ellipses.

We use the same notation of paper \cite{1010.5338}, Proposition 5.6, with minor modifications. For convenience, we replicate it here: 
\begin{equation*}
S_1^{\pm}(a) = \left\{ \pm \frac{\sqrt3}{2}a \right\} \times \left[ -\frac{a}{2}, -\frac{a}{4} \right] .
\end{equation*}

Notice that $S_1^{+}(a)$ and $S_1^{-}(a)$ are both segments on $\RR^2$. We also define the similar segments 
$S_2^{\pm}(a)$ and $S_3^{\pm}(a)$. Denote by $\cR_{2\pi/3}$ the counter-clockwise rotation of angle $2\pi/3$ 
around the origin on $\RR^2$ and define $S_j^{\pm}(a) = \cR_{2\pi/3}^{j} S_1^{\pm}(a)$, for $j = 2, 3$. If for each $j$ we have an ellipse connecting $S_j^{+}(a)$ and $S_j^{-}(a)$, then we have formed a circuit of ellipses around the origin (see Figure \ref{fig:TW_circuit_of_ellipses}).

Our objective is to check for which values of $\alpha$ and $u$ we can guarantee that these circuits will appear infinitely often. Fortunately, the proof holds even in the case $\alpha = 2$ and thus Lemma \ref{lema:alpha_le_2_is_subcritical} below will be used also in the proof of Theorem \ref{teo:vacant_crossing_of_boxes_alpha_2}.

\begin{lema}
\label{lema:alpha_le_2_is_subcritical}
Fix $\alpha \le 2$. Then, for any $\rho$ with decay $\alpha$ we have $\bar{u}_c(\rho) = 0$.
\end{lema}

\begin{proof}
We use Proposition \ref{prop:crossing_box_with_one_ellipse} for a fixed proportion of the box we would like to cross. Consider the box $B_1(a) := \bigl[ - \tfrac{\sqrt3}{2}a, \tfrac{\sqrt3}{2}a \bigr] \times \bigl[ -\tfrac{a}{2}, -\tfrac{a}{4} \bigr]$, which is a translation of the box $B_\infty(\tfrac{a}{4} ; 4\sqrt{3})$. Also, define $B_j(a)$ for $j = 2, 3$ by rotating the already defined box $B_1(a)$. Notice that the events
\begin{equation*}
C_j(a) := \left\{ \exists s \in \supp \xi; \ E(s) \cap S_j^+(a) \neq \varnothing \text{ and } E(s) \cap S_j^-(a) \neq \varnothing \right\}
\end{equation*}
are not independent for different $j$. To get independence, we restrict ourselves to ellipses centered on smaller boxes contained on $B_j(a)$, exactly like we did on the proof of the lower bound of Proposition \ref{prop:crossing_box_with_one_ellipse}. Recall that in the proof of the lower bound we considered 
the event $LR_1^-(l;k)$ in which our ellipses had to be centered on $B_\infty(l/2; k)$. However, the choice of $1/2$ was arbitrary and if we consider only ellipses centered on $B_{\infty}(cl;k)$ for some fixed $c \in (0,1)$, we obtain the same lower bound with a different constant $\uc{const:crossing_box_with_one_ellipse}$.

We just have to choose some constant $c \in (0,1)$ to force the ellipses that make the crossing of $B_j(a)$ to have their centers in disjoint regions. Define $D_1(a) := [-ca , ca] \times [ h_1 a, h_2 a]$, where the constants $c$, $h_1$ and $h_2$ are chosen so that 
\begin{equation*}
D_1(a) \subset B_1(a) \ \text{and} \ D_1(a) \cap B_i(a) = \varnothing, \  \forall i \neq 1.
\end{equation*}

\begin{figure}
\centering
\begin{tikzpicture}[scale=.25]
 \draw (0,0) circle ({20/sqrt(3)});
 \draw[thick,->] (-15,  0) -- (15, 0);
 \draw[thick,->] (  0,-13) -- ( 0,15);
 
 \foreach \x in {0, 120, 240}{
 \begin{scope}[rotate= \x]
 \draw              (-10,{-10/sqrt(3)}) -- ( 10,{-10/sqrt(3)});
 \draw[ultra thick] ( 10,{-10/sqrt(3)}) -- ( 10,{- 5/sqrt(3)});
 \draw              ( 10,{- 5/sqrt(3)}) -- (-10,{- 5/sqrt(3)});
 \draw[ultra thick] (-10,{- 5/sqrt(3)}) -- (-10,{-10/sqrt(3)});
 
 \filldraw[pattern=north east lines,pattern color=gray!80]     (- 3,{- 9/sqrt(3)}) rectangle (  3,{- 6/sqrt(3)});
 \end{scope}
 }
 
 \node            at (9.7,{-3.7/sqrt(3)}) {\small $B_1(a)$};
 \node[right=1mm] at ( 10,{-7.5/sqrt(3)}) {\small $S^+_1(a)$};
 \node[ left=1mm] at (-10,{-7.5/sqrt(3)}) {\small $S^-_1(a)$};
 \node            at (  0,{-7.5/sqrt(3)}) {\small $D_1(a)$};
 
\end{tikzpicture}
\caption{Boxes $B_j(a)$ and $D_j(a)$ for $j \neq 1$ are obtained by rotations.}
\label{fig:TW_circuit_of_ellipses}
\end{figure}
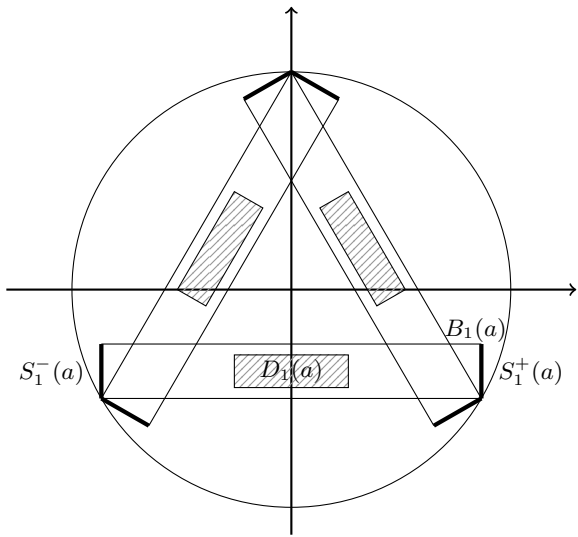

As we did before, we can use rotations to define the analogous regions $D_j(a)$ for $j = 2, 3$. Using the regions $D_j$, we can define events that are similar to $C_j(a)$ and are independent indeed. Define
\begin{equation*}
\tilde{C}_j(a) := \left\{ \exists s \in \supp \xi; 
 \begin{array}{l}
          E(s) \cap S_j^+(a) \neq \varnothing, E(s) \cap S_j^-(a) \neq \varnothing \ \text{and} \ c(E(s)) \in D_j(a)
     \end{array}
 \right\}.
\end{equation*}

By the proof of the lower bound in Proposition \ref{prop:crossing_box_with_one_ellipse}, we get the bound:
\begin{equation}
 P[\tilde{C}_j(a)] \geq 1 - \exp[- \uc{const:crossing_box_with_one_ellipse}^{-1} u a^{2 - \alpha}], \ \forall j.
\label{eq:low_bound_indep_crossing_C_tilde_j}
\end{equation}

Define $\Delta_a := \cap_{j = 1}^{3} \tilde{C}_j(a)$. Using the bound on equation \eqref{eq:low_bound_indep_crossing_C_tilde_j}, we have
\begin{align*}
P[\Delta_a] 
  &=  P[\tilde{C}_1(a) \cap \tilde{C}_2(a) \cap \tilde{C}_3(a)] \stackrel{\text{(Indep.)}}{\geq} (1 - \exp[- \uc{const:crossing_box_with_one_ellipse}^{-1} u a^{2 - \alpha}])^3.
\end{align*}

Finally, notice that taking the sequence $a_n = 3^n$ makes the events $\Delta_{a_n}$ independent, since they only depend on what the realization of the PPP $\xi$ looks like on disjoint regions of $\RR^2$. Thus, by Borel-Cantelli's lemma, since 
\begin{equation*}
\sum_{n \geq 1} P[\Delta_{a_n}] \geq \sum_{n \geq 1} (1 - \exp[- \uc{const:crossing_box_with_one_ellipse}^{-1} u 3^{n(2 - \alpha)}])^3 = \infty \text{ for $\alpha \le 2$,}
\end{equation*}
we conclude that $P[\Delta_{a_n}, \ \text{i.o.}] = 1$ and then $P[\cV \ \text{percolates}] = 0$.
\end{proof}

\subsection{Proving Phase Transition in \texorpdfstring{$u$ for $\alpha > 2$}{u for alpha greater than 2}}
\label{subsection:proving_phase_transition_in_u_for_alpha_greater_2}

The last ingredient to finish the proof of Theorems \ref{teo:alpha_phase_transition} and \ref{teo:covered_set_alpha_phase_transition} is to prove the behavior of ellipses model when $\alpha > 2$. As we mentioned in the introduction, this can be easily done by dominating $(u, \rho)$-ellipses model by Boolean model of radius distribution $\rho$ and intensity $u \lambda$, because of the results of Gou{\'e}r{\'e} \cite{gouere2008subcritical}. The techniques in \cite{gouere2008subcritical} are enough to prove the existence of a phase transition in $u$ for this values of $\alpha$ for both $\cE$ and $\cV$. However, studying the vacant set was not a priority in \cite{gouere2008subcritical}. For convenience of the reader we provide a full argument for this case, through Lemmas \ref{lema:bound_on_qk_in_multiscale_renorm} and \ref{lema:phase_transition_in_u_for_alpha_greater_2}.

Denote by $P_{u,\rho}^{\circ}$ the probability measure associated to the Boolean model above defined. Notice that since $\rho$ has tail decay $\alpha$ and support on $[1,\infty)$, we have
\begin{equation}
E[R^t] = \int_1^\infty R^t \, \rho(\text{d}R) =  \int_1^\infty \int_0^R ty^{t-1}\, \text{d}y \, \rho(\text{d}R) = \int_0^\infty ty^{t-1}\, \rho(R \geq y) \, \text{d}y
\end{equation}
which implies $E[R^t] < \infty$ for $t \in (0, \alpha)$. Since $\alpha > 2$, we have that $E[R^2] < \infty$ and thus by Theorem 2.1 of \cite{gouere2008subcritical} there is a positive constant $c$ such that  
\begin{equation*}
P_{u, \rho}[\cE \ \text{percolates}] \le P_{u, \rho}^{\circ}[\cE \ \text{percolates}] = 0, \ \forall u \in (0,\, cE[R^2]^{-1}).
\end{equation*}

Moreover, by Theorem 2.2 of \cite{gouere2008subcritical} if we define $\cC$ as the connected (covered) component of the origin and $D:= \diam \cC$ then for any fixed $t \in (0, \alpha - 2)$ we have
\begin{equation*}
E[R^{2 + t}] < \infty \ \text{implies} \ E_{u,\rho}^\circ [D^{t}] < \infty \ \text{for $u < c E[R^2]^{-1}$}
\end{equation*}
and by Markov's inequality we conclude $P_{u, \rho}[D \geq l] \le P_{u, \rho}^\circ[D \geq l] \le E_{u,\rho}^\circ [D^{t}] \cdot l^{-t}$. This means that the probability of the origin being connected to $\partial B(l)$ decays at least polynomially in $l$. This provides the correct decay of $P_{u,\rho}[0 \overset{\cE}{\longleftrightarrow} \partial B(l)]$, since by Proposition \ref{prop:new_ellipse_touches_ball} we have
\begin{equation*}
P_{u, \rho}[0 \overset{\cE}{\longleftrightarrow} \partial B(l)] \geq P_{u, \rho}[\exists s \in \supp \xi; z \in B(2l)^c, E(s) \cap B(l) \neq \varnothing] \geq 1 - \exp[- u \uc{const:lemma_ellipse_touches_ball_from_far} l^{2 - \alpha}] \sim u \uc{const:lemma_ellipse_touches_ball_from_far} l^{2 - \alpha}
\end{equation*}
when $l \to \infty$. The only statement we still have not proved in Theorems \ref{teo:alpha_phase_transition} and \ref{teo:covered_set_alpha_phase_transition} is that for small $u$ the vacant set percolates. We try to keep the same notation of \cite{gouere2008subcritical}. Define 
\begin{equation*}
\pi(l) := P_{u, \rho}^\circ[\partial B(l) \overset{\cE}{\longleftrightarrow} \partial B(8l) \ \text{using only balls centered on $B(10l)$}].
\end{equation*}

We denote by $G(l)$ the event in the definition of $\pi(l)$. Proposition 3.1 of \cite{gouere2008subcritical} proves there is a constant $C > 0$ such that 
\newconstant{const:rec_gouere}
\begin{equation}
\label{eq:recurrence_from_gouere}
\pi(10l) \le C \pi(l)^2 + u C \int_l^\infty \!\! R^2 \, \rho(\text{d}R) \le C \pi(l)^2 + u \uc{const:rec_gouere} l^{2 - \alpha}, \forall l \geq 1,
\end{equation}
\noindent in which the last inequality follows from a straightforward computation and $\rho[l, \infty) \le \uc{const:R_decay} l^{-\alpha}$ and $\uc{const:rec_gouere} = \uc{const:rec_gouere}(\uc{const:R_decay}, \alpha)$ is a constant. Also, if $\partial B(l)$ is connected to $\partial B(8l)$ then either $G(l)$ happened or there is a ball centered on $B(10l)^c$ intersecting $B(8l)$. This leads to the bound
\begin{equation}
\label{eq:bounding_through_pi}
P_{u, \rho}^\circ[\partial B(l) \overset{\cE}{\longleftrightarrow} \partial B(8l)] \le \pi(l) + 1 - \exp\bigg[- u \int_{B(10l)^c} \hspace{-4mm} \rho\big[ |z| - 8l, \infty \big) \, \text{d}z\bigg]  \le \pi(l) + u\uc{const:rec_gouere} l^{2 - \alpha}
\end{equation}
by a computation similar to the one in equation \eqref{eq:recurrence_from_gouere}. Define $q_k(u, \rho) = P_{u, \rho}^\circ[\partial B(10^k) \overset{\cE}{\longleftrightarrow} \partial B(8 \cdot 10^k)]$ for $k \geq 0$. We have:
\begin{equation}
q_{k+1} \le \pi(10^{k+1}) + u\uc{const:rec_gouere}(10^{2 - \alpha})^{k + 1} \le C\pi(10^k)^2 + u\uc{const:rec_gouere}(10^{2 - \alpha})^{k} \le \uc{const:rec_gouere} q_{k}^2 + u\uc{const:rec_gouere}(10^{2 - \alpha})^{k}. \label{eq:recur_ineq_qk_gouere}
\end{equation}

Using the recurrence relation in \eqref{eq:recur_ineq_qk_gouere} we can prove that for small values of $u$ the sequence $q_k$ tends to zero very fast.

\begin{lema}
\label{lema:bound_on_qk_in_multiscale_renorm}
Fix $\alpha > 2$. There exists $u_0 = u_0(\alpha, \uc{const:R_decay}) > 0$ such that $q_k(u, \rho) \le \exp[ - 2(\alpha - 2) k ]$, for all $k \geq 1$ and for all $u < u_0$.
\end{lema}

\begin{proof}
Fix $\varepsilon = 2(\alpha - 2)$ and notice that $0 < \varepsilon < (\log 10)(\alpha - 2) $. After that, take $k_0 = k_0(\alpha, \uc{const:R_decay})$ sufficiently large so that
\begin{equation}
\uc{const:rec_gouere} \exp[ \varepsilon - \varepsilon k_0 ] < \frac{1}{2} \hspace{3mm} \text{and} \hspace{3mm} 
\uc{const:rec_gouere} \exp\left[ \left( \varepsilon - (\log 10)(\alpha - 2)  \right) k_0 + \varepsilon \right] < \frac{1}{2}.
\label{eq:bound_on_qk_preliminaries}
\end{equation}

The choices above are possible only because our previous choice of $\varepsilon$ and the fact that $\alpha > 2$ 
together imply the left hand sides on equation \eqref{eq:bound_on_qk_preliminaries} tend to zero when $k_0 \to \infty$. Now that we fixed $k_0$, let us choose $u_0$. Notice that $q_k$ must be increasing in $u$ and besides, 
\begin{equation*}
\lim_{u \to 0+} q_k(u) = 0, \ \text{for any fixed $k$.}
\end{equation*}

One way to see this is combining \eqref{eq:bounding_through_pi} and Lemma 3.6 of \cite{gouere2008subcritical}:
\begin{equation}
\label{eq:easy_bound_qk}
q_k(u) \le \pi(10^k) + u \uc{const:rec_gouere} (10^{2 - \alpha})^k \le u C 100^k + u \uc{const:rec_gouere} (10^{2 - \alpha})^k
\end{equation}

Thus, take $u_0 = u_0(\alpha, \uc{const:R_decay})$ sufficiently small such that $u_0 \le 1$ and $q_{k_0}(u_0) \le \exp [ - \varepsilon k_0 ]$. Proceeding by induction, we will extend this inequality for all $k \geq k_0$. 
Suppose $q_k(u_0) \le \exp [ - \varepsilon k ]$. Using equation 
\eqref{eq:recur_ineq_qk_gouere}, we have that
\begin{align}
\frac{q_{k+1}(u_0)}{\exp[- \varepsilon (k+1) ]} \ 
  &\le \ \uc{const:rec_gouere} q^{2}_{k} \exp[ \varepsilon (k+1) ] + u_0 \uc{const:rec_gouere} 10^{k(2 - \alpha)} \exp[ \varepsilon (k+1) ] \nonumber\\
  &\le \ \uc{const:rec_gouere} \exp[ -2 \varepsilon k +\varepsilon k + \varepsilon ] + u_0 \uc{const:rec_gouere} 10^{k(2 - \alpha)} \exp[ \varepsilon (k+1) ] \nonumber\\
  &=   \ \uc{const:rec_gouere} \exp[ -\varepsilon k +\varepsilon] + u_0 \uc{const:rec_gouere} \exp[ - (\log 10)(\alpha - 2)k + \varepsilon(k+1) ] \nonumber\\
  &=   \ \uc{const:rec_gouere} \exp[ -\varepsilon k +\varepsilon] + u_0 \uc{const:rec_gouere} \exp[ (\varepsilon - (\log 10)(\alpha - 2)) k + \varepsilon ].
\end{align}

Since the right-hand side of the last equation is decreasing in $k$, we can use $k_0$ in the place of $k$. But then, by our choice of $\varepsilon$, $k_0$ and $u_0$ we can conclude $q_{k+1}(u_0) \le \exp[- \varepsilon (k+1) ]$, completing the induction step. To extend the bound to values of $k$ smaller than $k_0$ we can simply decrease $u_0$ even more using the crude bound on \eqref{eq:easy_bound_qk}. Finally, since $q_k(u)$ is increasing in $u$ the bound is valid for all $u < u_0$.
\end{proof}

Using Lemma \ref{lema:bound_on_qk_in_multiscale_renorm} we can show that $P_{u,\rho}( \cV \ \text{percolates}) = 1$ for $u < u_0(\alpha, \uc{const:R_decay})$.

\begin{lema}
\label{lema:phase_transition_in_u_for_alpha_greater_2}
Fix $\alpha > 2$ and a constant $\uc{const:R_decay}$. Then, for any $\rho$ with tail decay $\alpha$ and associated constant $\uc{const:R_decay}$ there exists $u_0(\alpha, \uc{const:R_decay}) \in (0, \infty)$ such that $\bar{u}_c(\rho) \geq u_0$.
\end{lema}

\begin{proof}
Take $u_0(\alpha, \uc{const:R_decay})$ from Lemma \ref{lema:bound_on_qk_in_multiscale_renorm}. If under measure $P_{u_0, \rho}^{\circ}$ the set $\cV$ a.s. does not percolate then there must exist a sequence $\gamma_n$ of disjoint circuits around the origin with $\gamma_n \subset \cE$ and such that 
\begin{equation*}
\dist(0, \gamma_n \cap (\RR^+\times \{0\})) \to \infty.
\end{equation*}

Focusing on this observation, consider the sequence of balls $(B^i_j)_{1 \le i \le 32, j \geq 0}$ 
where (see Figure \ref{fig:circuits_surround_origin}): 

\begin{itemize}
 \item All $B^i_j$ have their centers on the horizontal axis and radius $10^j$.\vspace{-2mm}
 \item The ball $B^1_0$ has its center at point $(8,0)$.\vspace{-2mm}
 \item The balls $B^i_j$ and $B^{i+1}_j$ are adjacent with $B^{i+1}_j$ on the right $\forall j, \forall 1 \le i \le 31$.\vspace{-2mm}
 \item The boxes $B^{32}_j$ and $B^1_{j+1}$ are adjacent with $B^1_{j+1}$ on the right $\forall j$.
\end{itemize}

\begin{figure}
\centering
\begin{tikzpicture}[scale=.13]
 \draw (-1, 0) -- (82,0) node[above,pos=.96] {$...$};
 \node[above] at (32.5,0) {$...$};
 \draw (0,1) -- (0,-1) node[below] {$0$};
 \draw (10.5,1) -- (10.5,-1) node[below] {$8$};
 \clip (-2,-20) rectangle (80,20);
 
 \begin{scope}[shift={(10,0)}]
 \foreach \x in {0, 1, 2}{
   \foreach \y in {0,1,2,3,5,6,7,8}{
   \begin{scope}[shift={({9*(pow(3,\x) - 1)/2},{-pow(3,\x)/2})}]
    \draw ({ (\y + 0.5)* pow(3,\x)}, {0.5*pow(3,\x)}) circle ({pow(3,\x)/2});
   \end{scope}
   }
 } 
 \end{scope}
 
 \draw[thick] plot [smooth cycle] coordinates {(-10,0) (0,-18) (10,-15) (17,0) (10,8) (0,15)};
 \node at (6,15) {$\gamma_1$};
 \draw[thick] plot [smooth cycle] coordinates {(-20,0) (0,-40) (20,-30) (35,0) (20,45) (0,30)};
 \node at (35,15) {$\gamma_2$};
  \draw[thick] plot [smooth cycle] coordinates {(-30,0) (0,-60) (30,-45) (54,0) (30,40) (0,45)};
 \node at (52,15) {$\gamma_3$};
\end{tikzpicture}
\caption{If $\cV$ does not percolate, circuits $\gamma_n$ must intersect balls $B^i_j$ with arbitrarily large $j$.}
\label{fig:circuits_surround_origin}
\end{figure}
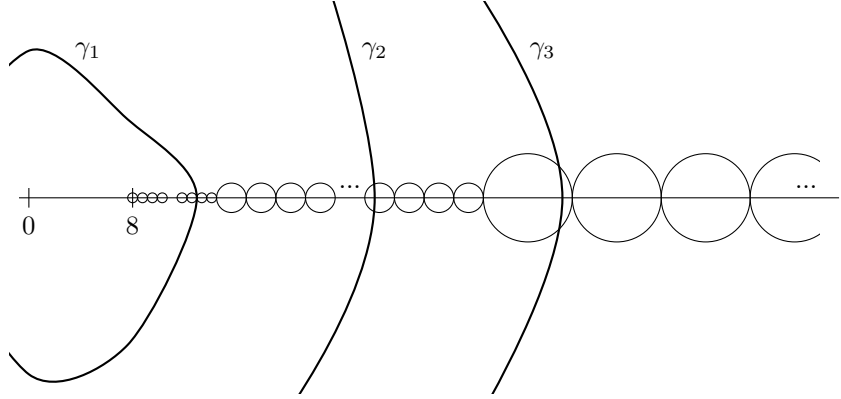

The choice of this construction of balls aims to ensure that whenever a circuit around the origin passes through $B^i_j$ a translation of the event $\{\partial B(10^j) \overset{\cE}{\longleftrightarrow} \partial B(8 \cdot 10^j)\}$ happens. Indeed, if we define $\tilde{B}^i_j := \{z \in \RR^2; \ \dist(z, B^i_j) \le 7\cdot 10^j \}$ then the event 
\begin{equation*}
A(B^i_j) = \{ \partial B^i_j  \overset{\cE}{\longleftrightarrow} \partial \tilde{B}^i_j\}
\end{equation*}
is just a translation of $\{\partial B(10^j) \overset{\cE}{\longleftrightarrow} \partial B(8 \cdot 10^j)\}$ and has also probability $q_j$. Notice that if there is a closed circuit around the origin $\gamma$ such that $\gamma \cap B^i_j \neq \varnothing$ then, since $\tilde{B}^i_j \subset \RR^+\times \RR$ by construction, we can deduce that $A(B^i_j)$ happened.

It follows from the definition of $B^i_j$ and the observation made above that if $\cV$ does not percolate then the circuits $\gamma_n$ must pass through balls $B^i_j$ with arbitrarily large $j$. Let $B_n$ be the ordering of balls $B^i_j$ sorted by their distance to the origin. Thus, using Lemma \ref{lema:bound_on_qk_in_multiscale_renorm} and Borel-Cantelli's Lemma, since $\sum_n 32 \cdot q_n < \infty$, we have
\begin{align*}
P_{u_0, \rho}\left( \{ \cV \ \text{perc.} \}^c \right) 
  &\le P_{u_0, \rho}^{\circ}\left( \{ \cV \ \text{perc.} \}^c \right) \le P_{u_0, \rho}^{\circ}\left( \begin{array}{l}
                  \text{$\exists (\gamma_n)$ circuits around origin with $\gamma_n \subset \cE$}\\
                  \text{and $\dist(0, \gamma_n \cap (\RR^+\times \{0\})) \to \infty$}
               \end{array}
        \right) \\
  &\le P_{u_0, \rho}^{\circ}[A(B_n), \ \text{i.o.}] = 0. \qedhere
\end{align*}
\end{proof}

\section{Decay of Correlations and Ergodicity}
\label{section:decay_of_correlations_with_distance}

In this section we derive a lemma that is useful to handle the dependence of some events in the ellipses model. It provides bounds that prove that some events are almost independent from one another if the distance between their dependence regions is large. In the same spirit of Section \ref{section:probability_of_simple_events}, this kind of estimate is essential to understand well any ellipses model.

One possible application of Lemma \ref{lema:decay_of_correlations_for_annulus} is to provide an alternative derivation of Lemma \ref{lema:bound_on_qk_in_multiscale_renorm}, without using reference \cite{gouere2008subcritical}. Moreover, we apply Lemma \ref{lema:decay_of_correlations_for_annulus} to prove ergodicity of the ellipses model with respect to the translations in $\RR^2$. Lemma \ref{lema:decay_of_correlations_for_annulus} is also important in our proof that when $\alpha = 2$ and $u$ is small the set $\cE$ does not percolate, almost surely.

Let $K$ be a measurable subset of $\RR^2$. Recall $\xi = \sum_i \delta_{s_i}$, where $s_i$ are points in $S$ for all $i$, is the PPP on $S$ that can be identified with the random collection of ellipses. Define 
\begin{equation*}
\xi_K := \sum_{i; E(s_i) \cap K \neq \varnothing} \delta_{s_i}
\end{equation*}
that is, the PPP obtained from $\xi$ by taking only the points of $\supp \xi$ whose ellipses intersect $K$.

\begin{defi}
 We say a function $f$ from the point processes on $S$ to $\RR$ depends only on the ellipses touching $K \subset \RR^2$ if $f(\xi) = f(\xi_K)$.
\end{defi}

We are now ready to state the decoupling we want to prove. The proof is similar to an argument of Sznitman \cite{sznitman2010vacant}, Theorem 2.1:

\newconstant{const:decay_of_correlations_annulus}
\begin{lema}
\label{lema:decay_of_correlations_for_annulus}
Take $\alpha > 1$. Let $K_1 = B(l_1)$ and $K_2 = B(l_2)^c$, with $l_2 = a l_1$, $a \geq 3$ and $l_1 \geq 1$. Let $f_1$ and $f_2$ be real functions of $\xi$ such that $|f_j| \le 1$, $f_1$ depends only on ellipses touching $K_1$ and $f_2$, on $K_2$. Then, there is a constant $\uc{const:decay_of_correlations_annulus} = \uc{const:decay_of_correlations_annulus}(\uc{const:R_decay}, \alpha) > 0$ such that 
\begin{equation}
\bigl| E[f_1 f_2] - E[f_1]E[f_2] \bigr| \le  u \uc{const:decay_of_correlations_annulus} l_1^{2-\alpha} (a - 1)^{1 - \alpha}.
\label{eq:decay_of_correlations_for_annulus}
\end{equation}
\end{lema}

\begin{proof}
Take two independent copies of $\xi$ and denote them by $\xi$ and $\xi'$. Fixed one of these copies, 
we decompose it into four independent PPP's on $\RR^4$. Consider the following partition of $S$:
\begin{align*}
\Gamma_1 &= \{ s; E(s) \cap K_1 \neq \varnothing \text{ and } E(s) \cap K_2 = \varnothing \}, &
\Gamma_2 &= \{ s;  E(s) \cap K_1 = \varnothing \text{ and } E(s) \cap K_2 \neq \varnothing \}, \\
\Gamma_{12} &= \{ s; E(s) \cap K_1 \neq \varnothing \text{ and } E(s) \cap K_2 \neq 
\varnothing \}, &
\Gamma_0 &= \{ s; E(s) \cap K_1 = \varnothing \text{ and } E(s) \cap K_2 = \varnothing \}. 
\end{align*}

These restrictions give birth to independent PPP's \cite{resnick2013extreme}. We decompose $\xi = \xi_1 + \xi_2 + \xi_{12} + \xi_0$, where $\xi_\square$ denotes the restriction of $\xi$ to the region $\Gamma_\square$. Analogously, we have $\xi' = \xi'_1 + \xi'_2 + \xi'_{12} + \xi'_0$. Define $\gamma_1 :=  \xi_1 + \xi'_2 + \xi_{12} + \xi'_0$ and $\gamma_2 :=  \xi'_1 + \xi_2 + \xi'_{12} + \xi_0$. Naturally, this construction makes $\gamma_1$ and $\gamma_2$ independent and with the same distribution of $\xi$. Besides, we have $(\gamma_1)_{K_1} = \xi_1 + \xi_{12} = \xi_{K_1}$ by construction. On the other hand, notice that $\xi_{K_2} = \xi_2 + \xi_{12}$ and $(\gamma_2)_{K_2} = \xi_2 + \xi'_{12}$. Using the relations above, we deduce:
\begin{align*}
P(  \xi_{K_2} \neq (\gamma_2)_{K_2} ) \ 
  & =  \ P(\xi_{12} \neq \xi'_{12}) \le P\left( \{\xi(\Gamma_{12}) \neq 0\} \cup \{\xi'(\Gamma_{12}) \neq 0\}\right) \le \ 2 P(\xi(\Gamma_{12}) \geq 1) \\
  &= 2\{ 1 - \exp[-(u\lambda \times \rho \times \nu)(\Gamma_{12})] \} \le \ 2 (u\lambda \times \rho \times \nu)(\Gamma_{12})
\end{align*}

\vspace{3mm}
\noindent\textbf{Estimating the measure of $\Gamma_{12}$:} Defining the notation $\mu := \lambda \times \rho \times \nu$, we want to estimate $\mu(\Gamma_{12})$. Notice that the distance between $K_1$ and $K_2$ is given by $(l_2 - l_1) =  l_1(a - 1)$. Thus, if $z \in \RR^2$ is the center of an ellipse intersecting both sets, we must have $R \geq l_1(a -1)/2$. To estimate the measure of $\Gamma_{12}$ we decompose $\RR^2$ into three different regions, according to the position of $z$. For $z \in B_1 := B(l_1 + 1)$, we have the trivial bound 
\begin{equation*}
\mu(z \in B_1, s \in \Gamma_{12}) \le \mu(z \in B_1, R \geq l_1(a -1)/2) \le \uc{const:decay_of_correlations_annulus} l_1^{2 - \alpha}(a - 1)^{-\alpha}.
\end{equation*}

If $B_2 := \{z; \ l_1 + 1 < |z| \le \frac{l_1 + l_2}{2} \}$ then for $z \in B_2$ we already have some restrictions on the possible values of $V$. Analogously to Lemma \ref{lema:prob_ellipse_intersects_ball_step1}, $V$ must be in an interval $V_{l_1,z}$ of total length $2\arcsin(\tfrac{l_1 + 1}{|z|})$ and thus
\vspace{-4mm}
\begin{align*}
\mu(z \in B_2, s \in \Gamma_{12}) 
  &\le \uc{const:decay_of_correlations_annulus} \!\int_{B_2} \int_{\tfrac{l_1(a-1)}{2}}^\infty \frac{l_1}{|z|} \, \rho(\text{d}R) \, \text{d}z \le \uc{const:decay_of_correlations_annulus} l_1 (l_1(a - 1))^{-\alpha} \! \int_{l_1 + 1}^{\tfrac{l_1 + l_2}{2}} \frac{1}{r} \, \text{d}r \le \uc{const:decay_of_correlations_annulus} l_1^{1 - \alpha}(a - 1)^{-\alpha}.
\end{align*}

Finally, if $z \in B_3 := \{z; |z| > \frac{l_1 + l_2}{2}\}$ then the restriction on $V$ still holds and now we use also $R \geq |z| - l_1$. We have
\vspace{-3mm}
\begin{align*}
\mu(z \in B_3, s \in \Gamma_{12}) 
  &\le \uc{const:decay_of_correlations_annulus} \int_{B_3} \int_{|z| - l_1}^\infty \frac{l_1}{|z|} \, \rho(\text{d}R) \, \text{d}z \le \uc{const:decay_of_correlations_annulus} l_1 \int_{\tfrac{l_1 + l_2}{2}}^{\infty} \left( 1 - \frac{l_1}{|z|} \right)^{-\alpha}|z|^{-(\alpha + 1)} \, \text{d}z \\
  &\le \uc{const:decay_of_correlations_annulus} l_1 \left( 1 - \frac{2l_1}{l_1 + l_2} \right)^{-\alpha} \int_{\tfrac{l_1 + l_2}{2}}^{\infty} |z|^{-(\alpha + 1)} \, \text{d}z = \uc{const:decay_of_correlations_annulus} l_1 \left( \frac{a + 1}{a - 1} \right)^{\alpha} \left[\frac{l_1 + l_2}{2}\right]^{-\alpha} \\
  &\le \uc{const:decay_of_correlations_annulus} l_1^{1 - \alpha} (a - 1)^{-\alpha}.
\end{align*}

Taking the worst of the three bounds gives $u\mu(\Gamma_{12}) \le u\uc{const:decay_of_correlations_annulus} l_1^{2 - \alpha}(a - 1)^{-\alpha}$.

\vspace{3mm}
Now, the only part that is still missing is how we relate the left hand side of equation 
\eqref{eq:decay_of_correlations_for_annulus} with the coupling we have defined above. Notice that
\begin{align}
E[f_1(\xi) f_2(\xi)] 
  & =  E[f_1(\xi_{K_1}) f_2(\xi_{K_2})] \nonumber\\
  & =  E\left[ f_1(\gamma_1) \ [f_2(\gamma_2) \ind{\xi_{K_2} = (\gamma_2)_{K_2}} + f_2(\xi_{K_2}) 
        \ind{\xi_{K_2} \neq (\gamma_2)_{K_2}}] \right] \label{eq:proof_decay_of_correlations_part_1}\\
\text{and also} \ \ \ E[f_1(\xi)] E[f_2(\xi)]
  & =  E[f_1(\gamma_1)] E[f_2(\gamma_2)]  \ \stackrel{\text{Ind.}}{=} \  E[f_1(\gamma_1) f_2(\gamma_2)] \nonumber\\
  & =  E\left[ f_1(\gamma_1) \ [f_2(\gamma_2) \ind{\xi_{K_2} = (\gamma_2)_{K_2}} + f_2(\gamma_2) 
        \ind{\xi_{K_2} \neq (\gamma_2)_{K_2}}] \right].
\label{eq:proof_decay_of_correlations_part_2}
\end{align}

Take the absolute value of the difference between the left hand sides in equations 
\eqref{eq:proof_decay_of_correlations_part_1} and \eqref{eq:proof_decay_of_correlations_part_2}. Using that 
$|f_j| \le 1$ and the triangular inequality, we get
\begin{align*}
\bigl| E[f_1 f_2] - E[f_1]E[f_2] \bigr| \ 
  &=    \ \left| E\left[ f_1(\gamma_1) ( f_2(\xi_{K_2}) - f_2(\gamma_2))\ind{\xi_{K_2} \neq (\gamma_2)_{K_2}} 
\right] \right| \nonumber\\
  &\le  \ 2 P[\xi_{K_2} \neq (\gamma_2)_{K_2}] \ \le \ u\uc{const:decay_of_correlations_annulus} l_1^{2 - \alpha}(a - 1)^{-\alpha}. \qedhere
\end{align*}
\end{proof}

The same method above can be used to estimate the decay of correlations for two balls. If we assume also that $\alpha > 2$, we are able to prove a bound that depends only on the distance of the balls and not their diameter.

\newconstant{const:decay_of_correlations}
\begin{lema}
\label{lema:decay_of_correlations}
Take $\alpha > 2$. Let $K_1$ and $K_2$ be (euclidean) balls with the same diameter $h$ and $r :=\nobreak \dist(K_1, K_2) \geq 2$. Let $f_1$ and $f_2$ be real functions of $\xi$ such that $|f_j| \le 1$ and $f_j$ depends only on ellipses touching $K_j$, for $j = 1, 2$. Then, there is a constant $\uc{const:decay_of_correlations} = \uc{const:decay_of_correlations}(\alpha, \uc{const:R_decay})$ such that 
\begin{equation}
\label{eq:decay_of_correlations}
\bigl| E[f_1 f_2] - E[f_1]E[f_2] \bigr| \le \uc{const:decay_of_correlations} u r^{2-\alpha}.
\end{equation}
\end{lema}

To end this section, we prove the ergodicity of the ellipses model for translations. Consider the family of translations $(\tau_x)_{x \in \RR^2}$, where $\tau_x: \RR^2 \ni v \mapsto v + x$. We already know that $P_{u, \rho}$ is invariant with respect to any $\tau_x$. With Lemma \ref{lema:decay_of_correlations_for_annulus} we can prove more:
\begin{lema}
\label{lema:erdodicity_ellipse_model}
Let $A$ be an event such that $\tau_x(A) = A$, $P_{u, \rho}$-a.s., $\forall x\in \RR^2$. Then, $P_{u, \rho}[A] \in \{0, 1\}$.
\end{lema}

\begin{proof}
We omit $u$ and $\rho$ from $P_{u, \rho}$. Consider a sequence of events $A_n$ such that $P[A_n \Delta A] \le 2^{-n}$ and depends only on ellipses touching a finite ball $B(r_n)$. Then, for any fixed $n$ we have
\begin{align*}
|P[A] - P[A]^2| = |E[\inds{A}\inds{\tau_x(A)}] - E[\inds{A}]^2| 
  &\le |E[\inds{\tau_x(A)}(\inds{A} - \inds{A_n})]| + |E[\inds{A_n}(\inds{\tau_x(A)} - \inds{\tau_x(A_n)})]| + \\
  &\hspace{5mm}+ |E[\inds{A_n}\inds{\tau_x(A_n)}] - E[\inds{A_n}]^2| + |E[\inds{A_n}]^2 - E[\inds{A}]^2|
\end{align*}

If we make $x \to \infty$ then the term $|E[\inds{A_n}\inds{\tau_x(A_n)}] - E[\inds{A_n}]^2| \to 0$ by Lemma \ref{lema:decay_of_correlations_for_annulus} with $f_1 = \inds{A_n}$ and $f_2 = \inds{\tau_x(A_n)}$. Meanwhile, all the other terms can be bounded by $2P[A \Delta A_n]$, uniformly in $x$. Thus, making $x \to \infty$ and then $n \to \infty$ we obtain $P[A] = P[A]^2$.
\end{proof}

\section{Vacant Crossing of Boxes for \texorpdfstring{$\alpha = 2$}{alpha equal 2}}
\label{section:vacant_crossing_of_boxes_for_alpha_equal_2}

By Theorems \ref{teo:alpha_phase_transition} and \ref{teo:covered_set_alpha_phase_transition}, we already know that there is a phase transition in $\alpha$ for the percolative behavior of $\cV$ and $\cE$. In the process, we discovered with Proposition \ref{prop:crossing_box_with_one_ellipse} that when $\alpha = 2$ the probability of one ellipse crossing $B_\infty(l;k)$ is bounded away from 0 and 1, a very curious property.

By this reason, we focus on $\alpha = 2$ and study the vacant crossing of boxes. We fix $\alpha$ as 2 in this whole section. We begin defining the event in which we have a vacant crossing of a box $B_\infty(l;k)$ and restating Theorem \ref{teo:vacant_crossing_of_boxes_alpha_2}.

\begin{defi}
\label{defi:vacant_crossing}
Define the event:
\begin{equation*}
\overline{LR}(l;k) := \left\{ \exists \gamma: [0,1] \to \RR^2; 
     \begin{array}{l}
          \text{$\gamma$ is continuous,} \ \gamma([0,1]) \subset \cV \cap B_\infty(l;k), \\
          \gamma(0) \in L^-(l;k) \ \text{and} \ \gamma(1) \in L^+(l;k)
     \end{array} 
 \right\}
\end{equation*}
\end{defi}

\begin{teo:vacant_crossing_of_boxes_alpha_2}
Let $\rho$ be a distribution with $\alpha = 2$. Then, there exists $\bar{u} = \bar{u}(\uc{const:R_decay}) > 0$ such that for any fixed $k > 0$, $u \in (0,\bar{u})$ and $l > 0$ 
\begin{equation}
\label{eq:vacant_crossing_bounded_away}
\delta \le P_{u, \rho}[\overline{LR}(l;k)]\le 1 - \delta,
\end{equation}
where $\delta = \delta(\uc{const:R_decay}, u, k) > 0$. Moreover, for $u \in (0, \bar{u})$ we have:
\begin{equation}
\label{eq:cV_and_cE_not_percolate}
P_{u, \rho}[\text{neither} \ \cV \ \text{nor} \ \cE \ \text{percolate}] = 1.
\end{equation}
\end{teo:vacant_crossing_of_boxes_alpha_2}

\begin{proof}
We begin proving the upper bound on \eqref{eq:vacant_crossing_bounded_away}. Notice that by duality the event $\overline{LR}(l;k)$ is the complementary event of the one in which there is a vertical covered crossing of box $B_\infty(l;k)$ and if we apply a rotation by $\tfrac{\pi}{2}$ we get a left-right covered crossing of the box $B_\infty(kl;\tfrac{1}{k})$. Thus, using rotational invariance of the model and Proposition \ref{prop:crossing_box_with_one_ellipse} we can deduce
\begin{align}
P[\overline{LR}(l;k)] 
  &= P[LR(kl;\tfrac{1}{k})^c] \le P[LR_1(kl;\tfrac{1}{k})^c] \le \exp[ - \uc{const:crossing_box_with_one_ellipse}^{-1} u (k^{-1} \wedge k^{2})]
  \label{eq:up_bound_for_barLR}
\end{align}
for a constant $\uc{const:crossing_box_with_one_ellipse} = \uc{const:crossing_box_with_one_ellipse}(\uc{const:R_decay})$ whenever $l > 2$. Choosing $\delta(\uc{const:R_decay}, u,k)$ accordingly we get $P[\overline{LR}(l;k)] \le 1 - \delta$ for $l > 0$. However, for the lower bound we will need to choose $\bar{u}(\uc{const:R_decay})$ with some care.

%

\vspace{2mm}
\noindent\textbf{Idea for the lower bound:} Let us discuss how to prove the lower bound. A short argument (Lemma \ref{lema:vacant_crossing_of_boxes_alpha_2_k_2} below) implies that we only need to study the case $k = 2$ and $l > 1$. After that, all we have to do is to build an event that implies the vacant crossing of $B_\infty(l;2)$ and whose probability we can bound more easily. For that reason, we start to decompose the PPP $\xi$ into a sum of independent PPPs and analyze their contributions to the event we want to study. 

The first simplification is obtained through Proposition \ref{prop:new_ellipse_touches_ball}.\textit{(i)}. As we stressed before on Remark \ref{remark:ellipse_touches_ball_from_far}, it proves that we can pay a small price to prevent interference from ellipses centered too far away from our region of interest.

The second simplification is to notice that, since we can now worry only about ellipses centered in a finite ball, we can pay a reasonable price to ensure there are no ellipses with very large major axis inside it.

Finally, we have to treat the ellipses with major axis not too large. This is the most delicate part of the proof; since we are dealing with arbitrarily large scales, it is too much to expect that there will be no ellipses inside that region. We need to have some control on the connectivity even when there are many ellipses that could potentially block the vacant crossing. The idea behind this step is to use arguments of fractal percolation similar to the ones on Chayes, Chayes and Durrett \cite{ChayesDurrett1988fractal_percolation}, Theorem 1.

\vspace{3mm}
Following the script above, we begin simplifying our problem by restricting the values of $k$ and $l$ we need to analyze. We omit the proof of Lemma \ref{lema:vacant_crossing_of_boxes_alpha_2_k_2}, since it is a standard application of FKG inequality.

\begin{lema}
\label{lema:vacant_crossing_of_boxes_alpha_2_k_2}
If the lower bound holds for $k = 2$ then it holds for any $k > 0$.
\end{lema}

Because of Lemma \ref{lema:vacant_crossing_of_boxes_alpha_2_k_2} we know it is enough to prove the lower bound when $k = 2$, which we assume from now on. Notice that we can also assume $l$ is large. Indeed, we have the trivial bound
\begin{equation}
P[\, \overline{LR}(l;2) \,] \geq P[ B(2l) \subset \cV ] = P[B(2l) \cap \cE = \varnothing], \ \forall l > 0.
\label{eq:low_bound_barLR_l2_trivial}
\end{equation} 

The probability on the right hand side of equation \eqref{eq:low_bound_barLR_l2_trivial} is decreasing on $l$, so for $l \le 1$ we have by Proposition \ref{prop:new_ellipse_touches_ball}.\textit{(ii)} with $a = 2$ that $
P[\, \overline{LR}(l;2) \,] \geq P[ B(2) \cap \cE = \varnothing ] \geq \exp[ -u\uc{const:new_ellipse_touches_ball} 2^2 ]$.

Then, assume $k = 2$ and $l > 1$. The next step is to start the decomposition of $\xi$ into more treatable PPPs. For that, we introduce some new notation. We will partition the set $S$ into many parts, according to the ellipses positioning and size and also the sets they intersect. We recall that restricting a PPP to disjoint subsets gives birth to independent PPPs. 

\begin{defi}
\label{defi:subsets_for_partitioning_S}
Let $B, C \subset \RR^2$ and $D\subset \RR$. We define:
\begin{equation*}
\Gamma[B||C,D] := \{s \in S; \ \text{$E(s) \cap B \neq \varnothing$, $c(s) \in C$ and $R(s) \in D$}\}.
\end{equation*}
\end{defi}

\begin{defi}
\label{defi:restricted_PPP_xi}
We define $\xi_{B||C,D}$ as the restriction of PPP $\xi$ to the set $\Gamma[B||C,D]$. Also, we define the shorter versions $\xi_{B} := \xi_{B||\RR^2,\RR}$ and $\xi_{||C,D} := \xi_{\RR^2||C,D}$.
\end{defi}

\begin{remark}
Notice that the notation on Definition \ref{defi:restricted_PPP_xi} is consistent with Section \ref{section:decay_of_correlations_with_distance}.
\end{remark}

Our final objective is to use the decomposition of $\xi$ to build an event $H=H(l)$ with probability bounded away from zero for all $l$, and such that $H \subset \overline{LR}(l;2)$. To begin our partitioning, we restate Proposition \ref{prop:new_ellipse_touches_ball}.\textit{(i)} with the notation above.

\begin{lema}
\label{lema:restate_lemma_ellipse_touches_ball_from_far}
Let $w \in \RR^2$, $a \geq 1$ and $C \geq 2$. Then, we have
\begin{equation}
\exp[- u\uc{const:lemma_ellipse_touches_ball_from_far}^{-1} a^{2-\alpha}] \le P[\xi_{B(w,a)||B(w,Ca)^c, \RR}(S) = 0] \le \exp[- u\uc{const:lemma_ellipse_touches_ball_from_far} a^{2-\alpha}]
\end{equation}
for some constant $\uc{const:lemma_ellipse_touches_ball_from_far} = \uc{const:lemma_ellipse_touches_ball_from_far}(\uc{const:R_decay}, \alpha, C) > 0$.
\end{lema}

Notice that for any value of $l$ we have $B_\infty(l;2) \subset B(2l) \subset B(4l)$. Taking $\alpha = 2$, $w$ as the origin, $C = 2$ and $a = 2l$ in Lemma \ref{lema:restate_lemma_ellipse_touches_ball_from_far}, we get
\begin{equation}
P[\xi_{B(2l)||B(4l)^c, \RR}(S) = 0] \geq \exp[- u\uc{const:lemma_ellipse_touches_ball_from_far}^{-1}]
\end{equation}
for some constant $\uc{const:lemma_ellipse_touches_ball_from_far} = \uc{const:lemma_ellipse_touches_ball_from_far}(\uc{const:R_decay}) > 0$. Notice that on this event we do not have to worry about ellipses too far away (centered on $B(4l)^c$) interfering with the vacant left-right crossing of $B_\infty(l;2)$, since none of them intersect the ball $B(2l)$.

So we can restrict ourselves to the ellipses centered on $B(4l)$. Now, being confined to a finite region, we can pay a price to avoid ellipses with too large major axis. More precisely, notice that
\begin{align*}
P[\xi_{||B(4l), [l/2,\infty)}(S) = 0] 
   &=    \exp[ - u \, \lambda(B(4l)) \, P[R \geq l/2]] \\
   &\geq \exp[ - u \cdot 16\pi l^2 \cdot \uc{const:R_decay} 2^2l^{-2}] \geq \exp[ -u \, c(\uc{const:R_decay}) ].
\end{align*}

Define $H_1 := \{\xi_{B(2l)||B(4l)^c, \RR}(S) = 0, \ \xi_{||B(4l), [l/2,\infty)}(S) = 0\}$. By independence, we have that
\begin{equation}
P[H_1] \geq \exp[ - u \, c(\uc{const:R_decay})]
\label{eq:making_subevent_barLR_l2}
\end{equation}
where the constants $c$ and $\uc{const:lemma_ellipse_touches_ball_from_far}$ have been combined into a new one. On $H_1$, the only ellipses that could prevent $\overline{LR}(l;2)$ from happening must be centered on $B(4l)$ with major axis size in $[1,l/2)$.

%

%
%
%
%
%
%
%
%

\vspace{2mm}
\noindent \textbf{Introducing Fractal Percolation}

\vspace{2mm}
We want to compare the model we are currently studying with fractal percolation, as defined in \cite{ChayesDurrett1988fractal_percolation}. Let us give some definitions and results from \cite{ChayesDurrett1988fractal_percolation}. Consider a square box $B = [0,l]^2$ (in the original paper the boxes had side length 1, but this makes no difference in what follows).

Given a parameter $p \in [0,1]$ and $N \in \NN$, $N \geq 2$, we do the following inductive procedure: we divide all boxes into $N^2$ equal boxes and let any of them remain in the process with probability $p$, independently. Define the remaining set after the $n$-th step by $A_n$. More formally, we can define the boxes
\begin{equation*}
B^n_{ij} := \left[ \frac{(i-1)l}{N^n}, \frac{il}{N^n}\right] \times \left[ \frac{(j-1)l}{N^n}, \frac{jl}{N^n}\right] \ \text{for $0 \le i,j \le N^n$}
\end{equation*}
and take iid. random variables $(\epsilon^n_{ij})^{n \in \NN}_{0 \le i,j \le N^n}$ with $\epsilon^n_{ij} \distr \text{Ber}(p)$. Setting $A_0 = B$, we can define 
\begin{equation*}
A_n = A_{n-1} \cap \Big( \bigcup_{i,j; \, \epsilon^n_{ij} = 1} B^n_{ij} \Big).
\end{equation*}

In an analogous way, we can study fractal percolation in sets different from $B$. In order to do that, we can consider unions of disjoint boxes of side $l$ and then make the same procedure in each one of them.

In \cite{ChayesDurrett1988fractal_percolation}, Theorem 1, the authors prove that for $p$ sufficiently close to 1 the set $A_\infty = \cap A_n$ connects the opposing sides of box $B$ with high probability. Let us define a similar process in the model we are working with.

\vspace{2mm}
\noindent\textbf{The Removal Process}

\medskip
We can relate our model with fractal percolation through the following procedure. We fix $N = 2$. Initially, we divide the interval $[1,l/2)$, into $n_0 = n_0(l)$ disjoint intervals
\begin{equation}
I_n = \left\{
 \begin{array}{ll}
   \left[\tfrac{l}{2^{n+1}}, \tfrac{l}{2^n} \right) & \text{if $1 \le n < n_0$} \\[1ex]
   \left[1, \tfrac{l}{2^n} \right)                  & \text{if $n = n_0$}
 \end{array}
\right.
\end{equation}
where $n_0$ is the only integer such that $l/2^{n_0 + 1} \le 1 < l/2^{n_0}$, or equivalently $n_0 = \lceil \log_2 l \rceil - 1$. The box $B_\infty(l;2)$ is composed of 2 square boxes of side $l$. For each $n$ we can partition each of them into $4^n$ boxes of side $l/2^n$. We denote by $(B^n_z)_{z \in \Lambda_n}$ the collection of boxes of this partition. Since we want to find a vacant path connecting $L^-(l;2)$ and $L^+(l;2)$, we will successively test, for $n$ ranging from 1 to $n_0$, which are the sub boxes that were not intersected by ellipses with major axis in $I_n$. We define
\begin{equation}
X^n_{z} := \indAlto{\xi_{B^n_{z}||\RR^2, I_n}(S) = 0}.
\end{equation}

For every $n$, the family $(X^n_z)_{z \in \Lambda_n}$ is a random field. Notice that the collection of random fields $(X^n_\cdot)_{n = 1}^{n_0}$ is independent. However, for a fixed $n$ the values of $X^n_z$ on this random field are not independent; if $X^n_z = 0$, we have that $B^n_z$ has been intersected by an ellipse and then it is more probable that one of its neighboring boxes has also been intersected. Now, the choice of intervals $I_n$ becomes clearer. Since the random field $X^n_z$ is only concerned with ellipses whose major axis is in $I_n$, we have that one single ellipse cannot intersect two boxes at distance greater than $2l/2^n$. This means the random field $X^n_z$ is $2$-dependent. Moreover, if we denote by $\tilde{B}^n_z$ the region $B^n_z$ enlarged by $l/2^n$, we have
\begin{equation*}
P[X^n_z = 1] \geq P[\xi(\tilde{B}^n_z \times I_n \times (-\tfrac{\pi}{2}, \tfrac{\pi}{2}]) = 0] \geq \exp[- uc(\uc{const:R_decay})] =: p(u).
\end{equation*}

By this, we can apply the results of Liggett, Schonmann and Stacey \cite{liggett1997domination}. We conclude that if $p(u)$ is sufficiently close to one there is a $\beta = \beta(p(u))$ such that for each $n$ we can find an independent random field $(Y^n_z)_{z\in \Lambda_n}$ that is dominated by $(X^n_z)_{z \in \Lambda_n}$ and has a product law with $P[Y^n_z = 1] = \beta$. Moreover, we can take $\beta \to 1$ when $p \to 1$. This domination is enough to complete our proof, since if we take $\bar{u} = \bar{u}(\uc{const:R_decay})$ sufficiently small then $\beta(u)$ will be close to one for $u \in (0, \bar{u})$. Hence, the $n_0$-th step of fractal percolation $A_{n_0}$ obtained through
\begin{equation}
\label{eq:def_n0_fractal_step}
A_0 = B_\infty(l;2) \ \text{and} \ A_n = A_{n-1} \cap \Big(\bigcup_{\substack{\scriptscriptstyle z \in \Lambda_n;\\ \scriptscriptstyle Y^n_z = 1}} B^n_z \Big)
\end{equation}

\vspace{-3mm}
\noindent will contain a crossing of $B_\infty(l;2)$ with probability close to one, by Theorem 1 of \cite{ChayesDurrett1988fractal_percolation}. Define the random subset $\tilde{A}_{n_0}$ in the same way as $A_{n_0}$, substituting $Y^n_z$ by $X^n_z$. Stochastic domination implies we also have a left-right crossing of $B_\infty(l;2)$ in the set $\tilde{A}_{n_0}$. Translating to ellipses intersection, this means we found a random path on $B_\infty(l;2)$ that is not intersected by any ellipse with major axis in $[1, l/2)$. Denoting this event by $H_2$ and noticing that $H_2$ is independent from $H_1$, we conclude
\begin{equation*}
P[\overline{LR}(l;2)] \geq P[H_1 \cap H_2] \geq P[H_1]P[H_2] \geq \delta(u) > 0, \ \forall u \in (0, \bar{u}(\uc{const:R_decay})),
\end{equation*}
finishing the proof of the first claim on Theorem \ref{teo:vacant_crossing_of_boxes_alpha_2}.

\medskip
\noindent\textbf{Covered set does not percolate for small intensities}
\medskip

Using \eqref{eq:vacant_crossing_bounded_away}, let us prove that for $u < \bar{u}(\uc{const:R_decay})$ neither $\cV$ nor $\cE$ percolate, almost surely. The proof for $\cV$ is already done in Section \ref{subsection:triviality_of_critical_points_when_alpha_in_1_2}; we already know that $\cV$ does not percolate for any $u > 0$. 

To prove $\cE$ does not percolate, we can combine equation \eqref{eq:vacant_crossing_bounded_away} and FKG inequality. It follows that with probability at least $\delta(u, \uc{const:R_decay}) > 0$ we have a vacant circuit around the origin on $B(3l)\backslash B(l)$. So, we can pick an increasing sequence $(L_n) \subset \RR^+$, define the event $A_n$ in which there is such a circuit on $B(3L_n)\backslash B(L_n)$ and try to apply Borel-Cantelli's lemma, since $\sum P(A_n) = \infty$.

However, the events $A_n$ are not independent. The idea is then to choose $L_n$ increasing sufficiently fast so that the events $A_n$ get almost independent quickly. We apply Lemma \ref{lema:decay_of_correlations_for_annulus} for the functions $\inds{A_n}$ when $\alpha = 2$. Let us denote $a_n := \tfrac{L_n}{L_{n-1}}$. From Lemma \ref{lema:decay_of_correlations_for_annulus} we can deduce that if $j > i$ and $\tfrac{L_j}{3L_i} \geq 3$ then
\begin{equation*}
|P[A_i \cap A_j] - P[A_i]P[A_j]| \le \frac{u\uc{const:decay_of_correlations_annulus}}{\tfrac{L_j}{3L_i} - 1} = \frac{3u\uc{const:decay_of_correlations_annulus}}{a_{i+1} \ldots a_j - 3} \le \frac{3u\uc{const:decay_of_correlations_annulus}}{a_{i+1} - 3}.
\end{equation*}

Then, we can use a generalization of Borel-Cantelli due to Ortega and Wschebor \cite{ortega1984sequence}. The result states that a sufficient condition for $P[A_n, \ \text{i.o.}] = 1$ is that $\sum_n P(A_n) = \infty$ and 
\begin{equation}
\label{eq:ortega_condition_for_borel_cantelli}
\liminf_n \frac{\sum_{1 \le i < j \le n} [P(A_i \cap A_j) - P(A_i)P(A_j)]}{\left[ \sum_{i=1}^n P(A_i) \right]^2} \le 0.
\end{equation}

It is straightforward to check that if we take $L_1 \geq 1$ and the sequence $a_n$ satisfying $a_n \geq 9$ for all $n$ and $\sum_n \tfrac{1}{a_n} < \infty$ then condition \eqref{eq:ortega_condition_for_borel_cantelli} is satisfied. We omit the details.
\end{proof}

\bibliographystyle{plain}
\bibliography{Ellipses_Percolation_paper}

\end{document}